\documentclass[12pt]{amsart}
\usepackage[left=3cm,top=3cm,right=3cm]{geometry}
\usepackage{bbm}
\usepackage{mathrsfs}
\usepackage[usenames,dvipsnames,svgnames,table]{xcolor}
\usepackage[pdftex,pagebackref,colorlinks=true,linkcolor=BrickRed,citecolor=OliveGreen,pdfstartview=FitH]{hyperref}
\usepackage{amssymb}

\usepackage{graphicx}
\usepackage{hyperref}
\usepackage{float}
\usepackage[normalem]{ulem}
\usepackage{verbatim}

\makeatletter
\newsavebox{\@brx}
\newcommand{\llangle}[1][]{\savebox{\@brx}{\(\m@th{#1\langle}\)}%
  \mathopen{\copy\@brx\kern-0.5\wd\@brx\usebox{\@brx}}}
\newcommand{\rrangle}[1][]{\savebox{\@brx}{\(\m@th{#1\rangle}\)}%
  \mathclose{\copy\@brx\kern-0.5\wd\@brx\usebox{\@brx}}}
\makeatother

\newtheorem{thm}{Theorem}[section]

\newtheorem{coro}[thm]{Corollary}
\newtheorem{lem}[thm]{Lemma}
\newtheorem{prop}[thm]{Proposition}

\theoremstyle{definition}

\newtheorem{defn}[thm]{Definition}

\newtheorem{nota}[thm]{Notation}

\theoremstyle{remark}

\newtheorem{remk}[thm]{Remark}

\renewcommand{\phi}{\varphi}

\newcommand{\crit}{{\rm crit}}

\newcommand{\Rbb}{ {\mathbb R}}

\newcommand{\Cbb}{ {\mathbb C}}

\newcommand{\Zcal}{ {\mathcal Z}}

\newcommand{\be}{\begin{enumerate}}
\newcommand{\ee}{\end{enumerate}}

\newcommand{\LH}[1]{\textcolor{blue}{#1}}

\newcommand{\Mk}{M_{\kappa}}
\newcommand{\Mm}{M_{-1}}

\newcommand{\Dk}{-\Delta_{\kappa}}
\newcommand{\Dm}{-\Delta_{-1}}

\newcommand{\Gk}{\nabla_{\kappa}}
\newcommand{\Gm}{\nabla_{-1}}

\newcommand{\Rk}{\rangle_{\kappa}}
\newcommand{\Rm}{\rangle_{-1}}
\newcommand{\Tk}{T_{\kappa}}
\newcommand{\Kk}{K_{\kappa}}
\newcommand{\DDk}{D_{\kappa}}
\newcommand{\geuc}{g_{\text{Euc}}}

\renewcommand{\paragraph}{\bigskip \noindent}

\title{Hot spots in domains of constant curvature}

\author{Lawford Hatcher}

\begin{document}

\begin{abstract}
    We prove constant-curvature analogues of several results regarding the hot spots conjecture in dimension two. Our main theorem shows that the hot spots conjecture holds for all non-acute geodesic triangles of constant negative curvature. We also prove that, under certain circumstances, on constant (positive or negative) curvature triangles, first mixed Dirichlet-Neumann Laplace eigenfunctions have no non-vertex critical points. Moreover, we show that each of these eigenfunctions is monotonic with respect to some Killing field. Finally, we show that for general simply connected polygons of non-zero constant curvature---with exactly one family of exceptions---second Neumann eigenfunctions of the Laplacian have at most finitely many critical points. 
\end{abstract}

\maketitle

\section{Introduction}

Rauch's hot spots conjecture states that for a given bounded, Lipschitz planar domain (i.e. a connected open set) $\Omega\subseteq \Rbb^n$, solutions of the Neumann heat equations with generic initial conditions have extrema tending toward $\partial\Omega$ as time tends to infinity. This problem may be reformulated in terms of eigenfunctions of the Laplacian. Let $u$ be a function corresponding to the smallest positive eigenvalue $\mu$ of the problem 
\[\begin{cases}
    -\Delta u=\mu u\;\;&\text{in}\;\;\Omega\\
    \partial_{\nu}u\equiv 0\;\;&\text{in}\;\;\partial\Omega.
\end{cases}\]
The equivalent statement is that for any such $u$, the extrema of $u$ occur only on $\partial\Omega$ (see \cite{rauch} or \cite{banuelosburdzy}). Since this conjecture was posed in 1974, it has been disproven in full generality, and some progress has been made toward determining in what geometric settings the statement actually holds: \cite{banuelosburdzy}, \cite{burdzywerner}, \cite{jn}, \cite{freitas}, \cite{atarburdzy}, \cite{miyamoto}, \cite{judgemondal}, \cite{kennedyrohleder}, \cite{pont}, etc. There has also been recent work on extensions of this problem to mixed Dirichlet-Neumann boundary conditions: \cite{me}, \cite{liyao}, \cite{aldeghirohleder}, \cite{me2}, and \cite{me3}. The main purpose of this paper is to extend some of these results to the spherical and hyperbolic settings in dimension two.\\
\indent Let $\kappa\in \Rbb$.  If $\kappa \leq 0$, let $\Mk$ denote the unique connected, simply connected, complete $2$-dimensional space form of constant Gaussian curvature equal to $\kappa$ (see, e.g., \cite{lee}). For $\kappa=-1$, $\Mk$ is the hyperbolic plane, and for $\kappa=0$, $\Mk$ is the Euclidean plane $\Rbb^2$. If $\kappa>0$, let $\Mk$ denote the open unit $2$-dimensional hemisphere of radius $1/\sqrt{\kappa}$. Our main theorem is

\begin{thm}\label{mainthm}
    Let $T\subseteq \Mk$ be a non-acute geodesic triangle with $\kappa<0$, and let $u$ be a second Neumann eigenfunction for $T$. Then $u$ has no critical points. In particular, the hot spots conjecture holds for these triangles. Moreover there exists a Killing field $X$ (see Section \ref{killingfields} below)  such that $Xu>0$ in $T$.
\end{thm}

\begin{remk}
    As we will prove in Section \ref{nodalsection}, an eigenfunction $u$ on a geodesic polygon $P\subseteq\Mk$ extends smoothly to the interiors of edges of $P$, and we may study critical points of this extension. Theorem \ref{mainthm} states that this extension has no critical points. In particular, we show that the extrema of $u$ are located at the acute vertices of $T$. Throughout the paper, we never consider vertices of a polygon to be critical points of $u$ even if they are extrema. 
\end{remk}

\begin{remk}
    In the Euclidean setting, second Neumann eigenfunctions of triangles are always monotonic in some direction (see the recent work of Chen-Gui-Yao \cite{yao1}). The last statement in Theorem \ref{mainthm} is an analogous statement for non-acute triangles of negative curvature. 
\end{remk}

A simple corollary of Theorem \ref{mainthm} is that, for these particular triangles, the second Neumann eigenspace is one dimensional. The same is true for all non-equilateral Euclidean triangles (see \cite{atarburdzy} and \cite{siudeja}).

\begin{coro}\label{simple}
    If $T$ is a non-acute triangle of constant negative curvature, then the second Neumann eigenvalue of $T$ is simple.
\end{coro}

\noindent We conjecture that Corollary \ref{simple} holds also for non-equilateral triangles $T\subseteq \Mk$ for any $\kappa\in \Rbb$, but we have been unable to prove this.\\
\indent A natural extension of Theorem \ref{mainthm} would be to make the same statement for geodesic triangles of positive curvature. There is exactly one obstruction to extending our proof to these triangles: a hemispherical analogue of Lemma \ref{mixedineq} below. We strongly believe that this lemma holds for hemispherical triangles, but the proof we provide does not have an obvious extension to these triangles. In particular, we have 
\begin{thm}\label{sphereanalogue}
    Assume that the statement of Lemma \ref{mixedineq} below holds for $\kappa>0$. Then the statements in Theorem \ref{mainthm} and Corollary \ref{simple} all hold for non-acute geodesic triangles in $\Mk$ with $\kappa>0$. 
\end{thm}
\noindent We do not explicitly prove Theorem \ref{sphereanalogue} since one can check that each argument used to prove Theorem \ref{mainthm} extends with no trouble to positive curvature if we assume that Lemma \ref{mixedineq} holds.\\
\indent Next we study hot spots for the mixed problem in which Dirichlet boundary conditions are imposed on some parts of the boundary while Neumann conditions are imposed on the rest of the boundary. In particular, let $P\subseteq \Mk$ be a geodesic polygonal domain with any $\kappa\in \Rbb$. Let $D,N\subseteq \partial P$ be a partition of the edges of $P$ with neither $D$ nor $N$ empty. Consider the eigenvalue problem

\begin{equation}\label{mixedeqn}
    \begin{cases}
        -\Delta u=\lambda u\;\;&\text{in}\;\;P\\
        u\equiv 0\;\;&\text{in}\;\;D\\
        \partial_{\nu}u\equiv 0\;\;&\text{in}\;\;N.
    \end{cases}
\end{equation}

The first eigenvalue $\lambda_1$ of this problem is positive, and solutions to the heat equation with these boundary conditions with generically chosen initial conditions will have extrema tending toward the extrema of an eigenfunction corresponding to $\lambda_1$. Note also that this eigenvalue is always simple since $P$ is connected, so the analysis for these boundary conditions is often simpler than in the Neumann case. Eigenfunctions corresponding to $\lambda_1$ do not vanish in $P$, so we usually assume that such functions are non-negative. For these eigenfunctions, our results hold in both positive and negative curvature.

\begin{thm}\label{mixed}
    Let $\kappa\in \Rbb$, and let $T\subseteq \Mk$ be a geodesic triangle. If $D$ equals a single edge of $T$ and the vertex opposite to $D$ has interior angle at most $\pi/2$, then each first mixed eigenfunction $u$ of (\ref{mixedeqn}) has no critical points, and the Neumann vertex of $T$ is the global maximum of $u$ if $u\geq 0$. If $D$ equals two edges of $T$ and both mixed vertices have interior angle at most $\pi/2$, then each first mixed eigenfunction $u$ of (\ref{mixedeqn}) has exactly one critical point, and it is located on the Neumann edge. In either case, there exists a Killing field $X$ such that $Xu>0$ in $T$.
\end{thm}

Whether the methods used to prove Theorems \ref{mainthm} and \ref{mixed} can be generalized to other polygons is currently unclear. However, following the analysis performed by Judge-Mondal in \cite{remarksoncritset}, we are able to prove one general fact about hot spots on constant-curvature polygons. 

\begin{thm}\label{finite}
    Let $\kappa\in\Rbb\setminus\{0\}$, and let $P\subseteq \Mk$ be a simply connected polygon. If $D$ is a non-empty connected union of edges\footnote{We allow for $P$ to have vertices of angle $\pi$.} of $P$, then each first mixed eigenfunction of (\ref{mixedeqn}) has at most finitely many critical points. If a second Neumann eigenfunction of $P$ has infinitely many critical points, then $P$ is a hemispherical triangle with two right-angled vertices. 
\end{thm}

\indent We now provide a brief outline of the paper. In Section \ref{geometry}, we provide a short review of spherical and hyperbolic geometry, including a discussion of multiple coordinate systems that will be used for various computations. Sections \ref{nodalsection} and \ref{vertexsection} are devoted to studying the nodal and critical sets of eigenfunctions on geodesic polygons, establishing various facts about the behavior of eigenfunctions in the interior of the domain, in the interiors of edges, and near vertices. We then establish in Section \ref{ineqsection} a particularly useful eigenvalue inequality that holds for negatively curved geodesic triangles. In Sections \ref{killingfields} and \ref{nodalderivsection}, we recall several facts about Killing fields on $\Mk$ and then study the functions obtained by applying these vector fields to eigenfunctions. We then prove Theorem \ref{finite} in Section \ref{finitesection}. The proofs of Theorems \ref{mainthm} and \ref{mixed} rely of domain perturbation arguments, so we establish several lemmas in Sections \ref{perttheory} and \ref{implicationsection} pertaining to the behavior of critical points under these perturbations as well as some results showing that the existence of certain critical points imply the existence of others. Following this, we will be ready to prove Theorems \ref{mainthm} and \ref{mixed}, which we carry out in Sections \ref{neumannsection} and \ref{mixedsection}, respectively.

\section{Constant-curvature geometry in two dimensions}\label{geometry}
Here we recall some facts from two-dimensional constant-curvature geometry and set some notation. Let $\kappa\in \Rbb$. Recall that $\Mk$ is a unique geodesic space, meaning that any two points in $\Mk$ are joined by a unique geodesic.\\
\indent Throughout the paper, we will make use of multiple coordinate systems for $\Mk$ that are each useful in various contexts. These coordinate systems are each well known; see Thurston's book \cite{thurston}. We will let $\geuc$ denote the Euclidean metric on $\Rbb^2$.\\
\indent To prove the main theorems of paper, we construct paths of domains with varying curvature. If $\kappa<0$, let $\Kk$ denote the open unit disk in $\Rbb^2$ with Euclidean radius equal to $|\kappa|^{-1/2}$. If $\kappa\geq 0$, set $\Kk=\Rbb^2$. Then $\Kk$ endowed with the following metric, expressed in polar coordinates, is isometric to $\Mk$:
\begin{equation}\label{kleinmetric}
    g_{\kappa}=\frac{1}{(1+\kappa r^2)^2}dr^2+\frac{r^2}{1+\kappa r^2}d\theta^2.
\end{equation}
The most useful property of $\Kk$ is that the geodesics in $\Kk$ are exactly the Euclidean straight lines. Thus, for a geodesic polygon $P\subseteq\Kk$ viewed as a subset of $\Rbb^2$, endowing $P$ with the family of metrics $\{g_{\tau}\}$ for $\tau$ near $\kappa$ on $P$ yields a family of geodesic polygons with varying curvature. Note also that, in this model, the angle between two lines meeting at the origin equals its Euclidean angle (though this is not true for angles away from the origin).\\
\indent Among the most commonly used models for $\Mk$ is the Poincar\'e disk. For $\kappa=-1$, this model is well-known. We extend the model for $\kappa\in (-4/3,4/3)$ in a non-standard way as follows. Let \[\ell(\kappa)=\frac{\kappa}{4-3|\kappa|}.\] For $\kappa<0$, as a set, let $\DDk$ denote the disk in $\Rbb^2$ centered at the origin with radius equal to $|\ell(\kappa)|^{-1/2}$. For $\kappa\geq 0$, set $\DDk=\Rbb^2$. Endow $\Dk$ with the Riemannian metric
\begin{equation}\label{poincaremetric}
    g_{\DDk}=\frac{3|\ell(\kappa)|+1}{(1+\ell(\kappa)r^2)^2}\geuc.
\end{equation}
One can check that $\DDk$ has constant Gaussian curvature equal to $\kappa$. For $\kappa<0$, $\DDk$ is thus isometric to $\Mk$. For $\kappa>0$, $\DDk$ is isometric to the sphere of radius $1/\sqrt{\kappa}$ minus one point, and we may identify $\Mk$ isometrically with the disk of (Euclidean) radius $1/\sqrt{\ell(\kappa)}$ centered at the origin in $\DDk$. One can also see that $\Mk$ is isometric to the upper half-plane intersected with $\DDk$. Moreover, Euclidean lines through the origin represent geodesics in $\DDk$. \\
\indent The final model of which we make use will be applied only in the case $\kappa=-1$, for which $\Mk$ is the hyperbolic plane. Let $H$ be the open set $\{(x,y)\in \Rbb^2\mid y>0\}$ endowed with the Riemannian metric $g_H=(dx^2+dy^2)/y^2$. This metric is conformal to the Euclidean metric, so the angle between two geodesics in $H$ equals the Euclidean angle between the corresponding curves (which is also true of the Poincar\'e disk). Moreover, in this model, the geodesics are exactly the vertical lines and semicircles orthogonal to the $x$-axis. \\
\indent We will make frequent use of the isometries of $\Mk$. It is well-known that the isometry group of $\Mk$ is transitive on the set of geodesics and that for any maximal geodesic $\gamma$, there exists a unique non-trivial reflection isometry $\phi_{\gamma}:\Mk\to\Mk$ fixing $\gamma$ pointwise.\\
\indent Besides the geodesics, we will also at times make use of curves equidistant to a geodesic. A curve $\eta$ is called \textit{equidistant} to a geodesic $\gamma$ if and only if for all $p,q\in \eta$, the geodesic distance from $p$ to $\gamma$ equals the geodesic distance from $q$ to $\gamma$. When $\kappa=0$, equidistant curves to a geodesic are exactly the lines parallel to it. When $\kappa>0$, geodesics are great circles, and equidistant curves to a great circle are the lines of latitude with respect to the great circle. When $\kappa=-1$, in the upper-half plane model $H$ of $\Mm$, the curves equidistant to the geodesic defined by the $y$-axis are exactly the Euclidean rays in $H$ emanating from the origin.\\
\indent We will often apply terms from Euclidean trigonometry to triangles in $\Mk$. In particular, a triangle $T\subseteq \Mk$ is called a \textit{right triangle} if it has a vertex of angle $\pi/2$. The edge of $T$ opposite to its right-angled vertex will be called the \textit{hypotenuse} of $T$, and the other edges will be called \textit{legs} of $T$.

\section{Nodal and critical sets of eigenfunctions}\label{nodalsection}

We now recall some well-known properties of Laplace eigenfunctions. Let $\Omega\subseteq \Mk$ be a domain, and suppose that $\phi:\Omega\to \Rbb$ is a non-zero function satisfying the eigenfunction equation $\Dk\phi=\mu\phi$ (not necessarily satisfying any specific boundary conditions). Adopting the notation of \cite{judgemondal}, we denote the zero-level set of $\phi$ by $\Zcal(\phi):=\phi^{-1}(0)$ and the set of critical points of $\phi$ by $\crit(\phi)$. We refer to $\Zcal(\phi)$ as the \textit{nodal set} of $\phi$ and to $\crit(\phi)$ as the \textit{critical set} of $\phi$. When $\phi$ (resp. $\nabla\phi$) extends to points on $\partial\Omega$, we allow $Z(\phi)$ (resp. $\crit(\phi)$) to include boundary points. We call points in the intersection $\Zcal(\phi)\cap\crit(\phi)$ the \textit{nodal critical points} of $\phi$. The first result we recall is Theorem 2.5 of \cite{cheng}:

\begin{lem}[Theorem 2.5 \cite{cheng}]\label{nodalcritical}
    The set of nodal critical points of $\phi$ is discrete, and $\Zcal(\phi)$ contains no isolated points. Let $p\in\Omega$ be a nodal critical point of $\phi$. Then there exists a neighborhood $U$ of $p$ such that $\Zcal(\phi)\cap U$ is the union of $n\geq 2$ twice continuously differentiable\footnote{Because the Riemannian metric on $\Mk$ is real-analytic, eigenfunctions are real-analytic, so one can actually prove that these arcs are analytic.} arcs that intersect only at $p$. The $2n$ angles formed by the union of these arcs are all equal. Near a point $p\in \Zcal(\phi)\cap\Omega$, $\phi$ takes on both positive and negative values.
\end{lem}
We next present a constant-curvature analogue of Lemma 2.4 of \cite{judgemondal}. We replace references to ``line segments" by references to ``geodesic segments" in $\Mk$. Their proof relies only on the analyticity of eigenfunctions and the flat analogue of Lemma \ref{nodalcritical}, so we omit the proof.

\begin{lem}\label{nodalintersectgeodesic}
    Let $\ell$ denote a geodesic in $\Mk$, and let $k$ be a connected component of $\ell\cap \Omega$. Then either $\phi$ vanishes identically on $k$, or the intersection $k\cap \Zcal(\phi)$ is discrete. 
\end{lem}

Similarly, Proposition 2.5 of \cite{judgemondal} relies only on the real-analyticity of Laplace eigenfunctions on Euclidean domains and the fact that, in coordinates, the Laplace operator equals a positive, $C^{\infty}$ multiple of $\partial_x^2+\partial_y^2$. Because eigenfunctions on $\Mk$ are also real-analytic and, using a conformal metric for $\Mk$, the Laplace operator is a positive, $C^{\infty}$ multiple of $\partial_x^2+\partial_y^2$, their same proof applies to give 

\begin{prop}\label{critset}
    Each connected component of $\crit(\phi)$ is either an isolated point or a properly embedded real-analytic arc (homeomorphic to either a circle or a line segment). 
\end{prop}

The above results allow us to study Neumann eigenfunctions of geodesic polygons in $\Mk$ up to the smooth boundary points as in Section 3 of \cite{judgemondal}. Let $P\subseteq\Mk$ be a geodesic polygon, and let $u$ be a Neumann Laplace eigenfunction on $P$. As discussed in Section \ref{geometry}, for each edge $e$ of $P$, there exists a unique isometric reflection of $\Mk$ fixing $e$ pointwise. Using these isometries and the Neumann boundary conditions, we can extend the eigenfunction $u$ to a neighborhood of (the interior of) each edge to be even with respect to an appropriate isometry. Thus, there exists some open set $\Omega$ containing $P$ such that $\partial\Omega\cap\partial P$ is the set of vertices of $P$ and to which $u$ extends to an eigenfunction $\overline{u}$ that is real-analytic. Let $\phi$ be a Laplace eigenfunction of $\Omega$.\footnote{In applications, we will take $\phi$ to be some derivative of a second Neumann eigenfunction.} By Lemma \ref{nodalcritical}, $\Zcal(\phi)$ is a union of properly immersed curves in $\Omega$.\footnote{Recall that properly immersed curves in an open set have endpoints only in the boundary of the set.} In the following proposition, suppose that $\overline{\alpha}$ is a properly immersed curve of $\Zcal(\phi)$, and let $\alpha$ be a connected component of $\overline{\alpha}\cap \overline{P}$, where $\overline{P}$ is the closure of $P$ in $\Mk$. Then we get a curved version of Proposition 3.2 of \cite{judgemondal}. Their proof is purely topological and applies directly to the case of constant-curvature polygons in $\Mk$.
\begin{prop}\label{possiblecomponents}
    $\alpha$ is either an isolated point in $\partial P$ or is an immersed arc. If $\alpha$ is not a loop or an isolated point, then its endpoints are contained in $\partial P$.
\end{prop}

\section{Eigenfunctions near vertices of polygons}\label{vertexsection}
As in the Euclidean case, we use a local separation of variables near vertices of constant-curvature polygons to obtain global information about eigenfunctions. In this section, we perform this separation of variables and use this to obtain an expansion of eigenfunctions that has the same structure as in the Euclidean case. Due to the similarities of the expansions, we usually refer the reader to \cite{judgemondal} for proofs of the results we state. To simplify computations, we restrict to $\kappa\in (-4/3,4/3)$ so that we can employ the Poincar\'e disk model $\DDk$ of $\Mk$.\\
\indent Suppose that $P\subseteq\Mk$ is a geodesic polygon. Working in $\DDk$ and applying an isometry if necessary, we may suppose that $P$ has a vertex of interior angle $\beta$ at the origin with one adjacent edge contained in the positive $x$-axis such that, near this edge, $P$ is contained in the upper half-plane. Using polar coordinates, the Laplacian is then given by 
\[\Dk=-\frac{(1+\ell(\kappa) r^2)^2}{3|\ell(\kappa)|+1}\Big(\partial_r^2+\frac{1}{r}\partial_r+\frac{1}{r^2}\partial_{\theta}^2\Big).\]
\indent Choose $\epsilon>0$ sufficiently small such that the sector \[S_{\epsilon}=\{(r,\theta)\mid 0<r<\epsilon,0<\theta<\beta\}\] is contained in $P$. Let $u$ be an eigenfunction of $\Dk$ satisfying Neumann conditions on $\partial S_{\epsilon}\cap \partial P$, and let $\mu$ be the eigenvalue associated to $u$. We call such a vertex a \textit{Neumann vertex} of $P$. Let $\nu=\pi/\beta$. By the analyticity of $u$ away from the vertices, for fixed $r\in (0,\epsilon)$, the restriction of $u$ to the circular arc $\{(r,\theta)\mid 0<\theta<\beta\}$ is an $L^2$ function, so we have the Fourier expansion
\[\sum_{n=0}^{\infty}u_n(r)\cos(n\nu\theta)\]
where \[u_n(r)=\frac{2}{\beta}\int_0^{\beta}u(r,\theta)\cos(n\nu\theta)d\theta.\] Since $u$ satisfies Neumann conditions on $\partial S_{\epsilon}\cap \partial P$ and is analytic in $r$, we can differentiate and integrate this sum term by term. Let \[f=f_{\kappa}(r)=\frac{2}{\beta}\cdot\frac{(1+\ell(\kappa) r^2)^2}{3|\ell(\kappa)|+1}.\] We then have 
\begin{align*}
    \mu u=\Dk u&=\frac{2}{\beta}\sum_{n=0}^{\infty}\Big(\int_0^{\beta}(\Dk u)(r,\theta)\cos(n\nu\theta)d\theta\Big)\cos(n\nu\theta)\\
    &=\sum_{n=0}^{\infty}\Big[-f\Big(\partial_r^2+\frac{1}{r}\partial_r\Big)u_n(r)-\frac{f}{r^2}\int_0^{\beta}u(r,\theta)\partial_{\theta}^2\cos(n\nu\theta)d\theta\Big]\cos(n\nu\theta)\\
    &=\sum_{n=0}^{\infty}-f\Big(\partial_r^2+\frac{1}{r}\partial_r-\Big(\frac{n
    \nu}{r}\Big)^2\Big)u_n(r)\cos(n\nu\theta).
\end{align*}
By the uniqueness of the Fourier coefficients, the functions $u_n(r)$ must satisfy the ordinary differential equation
\[u_n''(r)+\frac{1}{r}u'(r)+\frac{-fn^2\nu^2+\mu r^2}{fr^2}u_n(r)=0.\]
We can then apply the Frobenius method for second-order differential equations (see, for instance, Teschl's book \cite{teschl} pp. 116-120) to see that this equation has (square integrable) solutions of the form
\[u_n(r)=a_nr^{n\nu}\sum_{j=0}^{\infty}h_{n,j}r^j,\] where $a_n\in \Rbb$ and the sum is analytic and converges uniformly on compact subsets of $\Mk$. Moreover, one can check that $h_{n,j}\neq 0$ for all $n$ and even $j$ and that $h_{n,j}=0$ for all $n$ and odd $j$. We can thus rewrite this expansion as 
\begin{equation}\label{neumannexpansion}
 u(r,\theta)=\sum_{n=0}^{\infty}a_nr^{n\nu}g_{n\nu}(r^2)\cos(n\nu\theta)   
\end{equation}
where $g_{n\nu}$ is analytic with nowhere vanishing Taylor coefficients. Moreover, we re-normalize the functions $g_{n\nu}$ if necessary such that $g_{n\nu}(0)=1$ for all $n$ and $g_0'(0)<0$. This is an exact analogue of the expansion (6) in \cite{judgemondal}, and all of the proofs of Section 4 of \cite{judgemondal} apply immediately to the case of constant curvature. We restate these results here without proof.
\begin{nota}
    When we wish to emphasize the dependence of the coefficients $a_n$ in (\ref{neumannexpansion}) on the vertex $v$ and the eigenfunction $u$, we may write $a_n(u,v)$. Similarly, we may write $b_n(u,v)$ and $c_n(u,v)$ in the expansions (\ref{mixedexpansion}) and (\ref{dirichletexpansion}) below. 
\end{nota}
\begin{remk}
    Using geodesic polar coordinates, one can obtain a more explicit expression than (\ref{neumannexpansion}) in terms of associated Legendre functions (see, e.g., \cite{terras}). However, it is more convenient for our purposes that our expansion has the same form as (6) in \cite{judgemondal}.
\end{remk}
\begin{lem}[Lemma 4.3 in \cite{judgemondal}]\label{arcsnearvertex}
    If $\beta\in (0,\frac{\pi}{2})$, then there exists a neighborhood $U$ of $0$ such that 
    \begin{enumerate}
        \item If $a_0\neq 0$, then $U\cap u^{-1}(\{u(0)\})=\{0\}$.
        \item If $a_0=0$ and $a_1\neq 0$, then $U\cap u^{-1}(\{u(0)\})$ is a simple arc containing $0$.
    \end{enumerate}
    If $\beta\in (\frac{\pi}{2},\pi)$, then there exists a neighborhood $U$ of $0$ such that 
    \begin{enumerate}
        \item If $a_1\neq 0$, then $U\cap u^{-1}(\{u(0)\})$ is a simple arc containing $0$.
        \item If $a_0\neq 0$ and $a_1=0$, then $U\cap u^{-1}(\{u(0)\})=\{0\}$.
    \end{enumerate}
    If $\beta\in(0,\pi)$ and $a_0=a_1=0$, then there exists a neighborhood $U$ of $0$ such that $U\cap u^{-1}(\{u(0)\})$ consists of at least two simple arcs intersecting only at $0$.
\end{lem}
\begin{remk}
    We emphasize that the roles of $a_0$ and $a_1$ in Lemma \ref{arcsnearvertex} are reversed between the acute and obtuse angle cases. The reason for this---as can be seen in the proof of Lemma 4.3 in \cite{judgemondal}---is that for $\beta$ acute, $\nu>2$, so the smallest positive power of $r$ in (\ref{neumannexpansion}) is $2$, and the angular part of the term with this exponent is constant. In contrast, for obtuse angles, $1<\nu<2$, so the smallest positive power of $r$ in (\ref{neumannexpansion}) is $\nu<2$, and the angular part of the term with this exponent vanishes on the angle bisector of the vertex.
\end{remk}

\begin{lem}[Proposition 4.4 in \cite{judgemondal}]\label{noaccumatvertex}
    If $\beta\in (0,2\pi)\setminus\{\frac{\pi}{2},\pi,\frac{3\pi}{2}\}$, then there exists a neighborhood of $0$ containing no critical points\footnote{Recall that we do not consider vertices to be potential critical points.} of $u$. If $\beta=\frac{\pi}{2}$, then there is a neighborhood $U$ of $0$ such that either $u$ has no critical points in $U$ or $U\cap \crit(u)$ equals an edge of $U\cap S_{\epsilon}$.
\end{lem}

\begin{lem}[Proposition 2.1 in \cite{erratum}]\label{vertexmax}
    If $\beta\in(0,\frac{\pi}{2})$, then $a_0\neq 0$ if and only if $0$ is a local extremum of $u$. If $\beta\in(\frac{\pi}{2},\pi)$, then $a_0\neq 0$ and $a_1=0$ if and only if $0$ is a local extremum of $u$. In either case, if the vertex is a local extremum, then it is a local maximum (resp. minimum) if $a_0>0$ (resp. $a_0<0$).
\end{lem}

We can repeat the above analysis on vertices near which an eigenfunction $u$ satisfies mixed boundary conditions. Suppose now that on $S_{\epsilon}\cap \partial P$, $u$ satisfies Dirichlet boundary conditions on $x$-axis and Neumann conditions on the ray $\{\theta=\beta\}$. We call such a vertex a \textit{mixed vertex} of $P$. Then on $S_{\epsilon}$, we have the expansion 
\begin{equation}\label{mixedexpansion}    u(r,\theta)=\sum_{n=0}^{\infty}b_nr^{(n+1/2)\nu}g_{(n+1/2)\nu}(r^2)\sin((n+1/2)\nu\theta).
\end{equation}
There are mixed analogues of Lemmas \ref{arcsnearvertex} and \ref{noaccumatvertex} that are in fact simpler to prove than in the Neumann case.
\begin{lem}\label{mixedarcsatvertex}
    If $\beta\in (0,\pi)$ and $b_0\neq 0$, then there exists a neighborhood $U$ of $0$ such that $U\cap \Zcal(u)$ is contained in the $x$-axis and such that $U$ contains no critical points of $u$. If $b_0=0$, then $0$ is the endpoint of an arc in $Z(u)\cap P$.
\end{lem}
\begin{proof}
    The first statement follows directly from the expansion (\ref{mixedexpansion}). Suppose that $b_0=0$, and let $m$ be the smallest integer for which $b_m\neq 0$. Then (\ref{mixedexpansion}) implies
    \[u(r,\theta)=b_mr^{(m+1/2)\nu}\sin((m+1/2)\nu\theta)+O(r^{\min\{(m+3/2)\nu,(m+1/2)\nu+2\}}).\]
    Since $m\geq 1$, the second statement quickly follows.
\end{proof}

We can also expand eigenfunctions near \textit{Dirichlet vertices} (vertices near which the eigenfunction satisfies Dirichlet conditions):
\begin{equation}\label{dirichletexpansion}
    u(r,\theta)=\sum_{n=1}^{\infty}c_nr^{n\nu}g_{n\nu}(r^2)\sin(n\nu\theta).
\end{equation}
Note that the sum in this case begins at $n=1$ instead of $n=0$. The next lemma has a proof nearly identical to that of Lemma \ref{mixedarcsatvertex}.
\begin{lem}\label{dirarcsatvertex}
    If $\beta\in(0,\pi)$ and $c_1\neq 0$, then there exists a neighborhood $U$ of $0$ such that \[\Zcal(u)\cap U=\partial S_{\epsilon}\cap\partial P\] and such that $u$ has no critical points in $U$. If $c_1=0$, then $0$ is the endpoint of an arc in $\Zcal(u)\cap P$.
\end{lem}

\begin{prop}\label{regularity}
    Suppose that $P$ is a convex geodesic polygon in $\Mk$. Let $u$ be a mixed eigenfunction of $P$ satisfying Dirichlet conditions on some set of edges of $P$ and Neumann conditions on the other edges of $P$ (we may have full Dirichlet or Neumann conditions). If each mixed vertex of $P$ is non-obtuse, then $u$ is an element of the Sobolev space $H^2(P)$. 
\end{prop}
\begin{proof}
    As in Section \ref{nodalcritical}, we can extend $u$ by reflection to a neighborhood of each Neumann edge to be real analytic. Near a Dirichlet edge, a similar argument applies by extending the eigenfunction to be odd with respect to reflection over the edge. This shows that $u$ is locally in $H^2$ away from the vertices. Near the vertices, the statement can be seen by integrating derivatives of the expansions (\ref{neumannexpansion}), (\ref{mixedexpansion}), and (\ref{dirichletexpansion}) term by term.
\end{proof}

\section{Inequalities between small eigenvalues}\label{ineqsection}
In this section we prove several inequalities involving first mixed and second Neumann eigenvalues. Several of our main results are proved via contradiction by constructing test functions that violate these inequalities.\\
\indent Let us recall the variational formulation for the eigenvalue problem (\ref{mixedeqn}). For $\Omega$ and $D\subseteq\partial \Omega$ as in (\ref{mixedeqn}), let $H^1_D(\Omega)$ be the Sobolev space of $L^2(\Omega)$ functions whose first derivatives are square integrable and that vanish (in the sense of trace) on $D$. Let $\mathcal{V}_k$ denote the set of all $k$-dimensional subspaces of $H^1_D$. Then we have
\[\lambda_k^D=\min_{V\in \mathcal{V}_k}\max_{u\in V\setminus\{0\}}\frac{\int_{\Omega}|\Gk u|_{\kappa}^2}{\int_{\Omega}u^2},\] where $\kappa$ equals the curvature of $\Omega$. When the $k$th eigenvalue is simple, the unique minimizing subspace is the span of the first $k$ eigenfunctions, and the only maximizers within that subspace are the $k$-th eigenfunctions. We will be mostly concerned with first mixed and second Neumann eigenvalues (where the Neumann problem corresponds to $D=\emptyset$), for which we have the more specific formulas
\begin{equation}\label{mixedminmax}
    \lambda_1^D=\min_{u\in H^1_D(\Omega)\setminus\{0\}}\frac{\int_{\Omega}|\Gk u|_{\kappa}^2}{\int_{\Omega}u^2}
\end{equation}
and
\begin{equation}\label{neumannminmax}
    \mu_2=\min_{V\in \mathcal{V}_2}\max_{u\in V\setminus\{0\}}\frac{\int_{\Omega}|\Gk u|_{\kappa}^2}{\int_{\Omega}u^2}.
\end{equation}

Recall the following well-known fact about the spectrum of hyperbolic triangles. 

\begin{lem}\label{onequarter}
    Let $T$ be a geodesic triangle in $M_{-1}$. Then the second Neumann eigenvalue of the Laplacian on $T$ is greater than $\frac{1}{4}$.
\end{lem}
\begin{proof}
    See pp. 219-220 of \cite{buser}.
\end{proof}

\begin{lem}\label{mixedineq}
    Let $T$ be a geodesic triangle in $\Mk$ with $\kappa<0$. Let $e$ be an edge of $T$. The second Neumann eigenvalue of $\Dk$ is less than or equal to the first eigenvalue for the mixed problem with Neumann conditions on $e$ and Dirichlet conditions on $\partial T\setminus e$.
\end{lem}
\begin{proof}
    We prove the result for $\kappa=-1$. The statement then holds for other $\kappa<0$ by rescaling the metric. Let $\mu$ denote the second Neumann eigenvalue, and let $\lambda$ denote the first mixed eigenvalue for the aforementioned eigenvalue problem. Suppose toward a contradiction that $\lambda<\mu$, and let $\phi$ be a first mixed eigenfunction associated to $\lambda$. Suppose that $T$ is embedded in the upper half-plane model of $\Mm$ such that $e$ is contained in a vertical geodesic. Let $w(x,y)=y^s$ where $s=\frac{1}{2}+\frac{i}{2}\sqrt{4\mu-1}$. Then one can compute that $\Dm w=\mu w$. Note further that $w$ satisfies Neumann boundary conditions on $e$ and that $|\Gm w|_{-1}^2=|s|^2|w|^2$ pointwise. Because $T$ is a triangle and $\kappa<0$, Lemma \ref{onequarter} applies to give that $4\mu-1>0$, so we get \[\mu=s(1-s)=\Big(\frac{1}{2}+\frac{i}{2}\sqrt{4\mu-1}\Big)\Big(\frac{1}{2}-\frac{i}{2}\sqrt{4\mu-1}\Big)=s\overline{s}=|s|^2,\] which implies that $|\Gm w|_{-1}^2=\mu |w|^2$ pointwise. We then have, for any $a,b\in \Cbb$, 
    \begin{align*}
        \int_T|\Gm(a\phi+bw)|_{-1}^2&=|a|^2\int_T|\Gm\phi|_{-1}^2+2\text{Re}\Big(a\overline{b}\int_T\langle\Gm\phi,\Gm w\Rm\Big)+|b|^2\int_T|\Gm w|_{-1}^2\\
        &=|a|^2\lambda\int_T|\phi|^2+2\text{Re}\Big(a\overline{b}\int_T\langle\Gm\phi,\Gm w\Rm\Big)+|b|^2\mu\int_T|w|^2.
    \end{align*}
    Integration by parts gives 
    \[\int_T\langle\Gm\phi,\Gm w\Rm=\mu\int_T\phi\overline{w}+\int_{\partial T}\phi\partial_{\nu}\overline{w}=\mu\int_T\phi\overline{w}\] since $\phi$ vanishes on $\partial T\setminus e$ and $w$ satisfies Neumann conditions on $e$. Combining this with the above and the supposition that $\lambda<\mu$ gives 
    \[\int_T|\Gm(a\phi+bw)|_{-1}^2< \mu\Big(|a|^2\int_T|\phi|^2+2\text{Re}\Big(a\overline{b}\int_T\phi\overline{w}\Big)+|b|^2\int_T|w|^2\Big)=\mu\int_T|a\phi+bw|^2.\]
    Since $\phi$ vanishes on $\partial T\setminus e$ and $w$ does not, these functions are linearly independent. Since the above inequality is strict, this contradicts the variational formula (\ref{neumannminmax}) for $\mu$. 
\end{proof}

\begin{remk}
    The proof of Lemma \ref{mixedineq} is inspired by the proof of Theorem 3.1 of \cite{lotoroh}. Our proof actually applies to any geodesic polygon with constant negative curvature for which the second Neumann eigenvalue is at least one quarter. It would be interesting to know whether the statement of the lemma applies to other constant-curvature polygons as in the Euclidean case. 
\end{remk}

\section{Killing vector fields in spaces of constant curvature}\label{killingfields}
In many works involving the hot spots conjecture in Euclidean domains (e.g., \cite{jn}, \cite{judgemondal}, \cite{yao1}), one of the central ideas is that directional and rotational derivatives of eigenfunctions are also eigenfunctions. Viewing these differential operators as vector fields, the natural generalization of this technique to curved domains is to use Killing fields. These are vectors fields that are the infinitesimal generators of one-parameter families of isometries (see, for instance, \cite{lee} Problems 5-22 and 6-24). On any Riemannian manifold $(M,g)$ with a Killing field $X$, we therefore have that $-\Delta_g X=-X\Delta_g$ since the Laplacian commutes with pullbacks by isometries. Moreover, if $\phi$ is an eigenfunction of $-\Delta_g$ with eigenvalue $\lambda$, then we have \[-\Delta_g(X\phi)=X(-\Delta_g\phi)=\lambda X\phi,\]
so $X\phi$ also satisfies the eigenfunction equation for $-\Delta_g$ with eigenvalue $\lambda$ (though in general $X\phi$ does not satisfy any particular boundary conditions). In the following sections we will use functions of the form $Xu$ with $u$ a low-energy eigenfunction to construct test functions for the functionals in (\ref{mixedminmax}) and (\ref{neumannminmax}). These test functions will allow us to study nodal and critical sets of these eigenfunctions in more detail. In this section, we collect some facts about the Killing fields on $\Mk$. For $\kappa=0$, these are exactly the constant and roational vector fields, so we restrict attention to $\kappa\neq 0$. 
\subsection{Hyperbolic Killing fields}
For $\kappa< 0$, there are three distinct types of Killing fields corresponding to the classification of hyperbolic isometries, and we will make use of only two of these: loxodromic and elliptic vector fields. Since $\Mk$ and $M_{\tau}$ for $\kappa,\tau<0$ differ only by a scalar multiple of the metric, we state our results only for the hyperbolic plane $\Mm$. For $-1\neq \kappa<0$, similar statements and formulae hold.\\
\indent Recall that, in the upper half-plane $H$, a \textit{loxodromic} isometry\footnote{Loxodromic isometries are often referred to as \textit{hyperbolic} isometries in the literature.} is an isometry conjugate to the map $(x,y)\mapsto (a x,a y)$ for some $a>0$. The infinitesimal generator of the latter family of isometries is given by the vector field \[L:=x\partial_x+y\partial_y.\] Because these isometries map the $y$-axis to itself, $L$ is tangential to the $y$-axis (this can also be seen directly from the formula for $L$). Moreover, $L$ is tangent to the equidistant curves to the $y$-axis (see Section \ref{geometry}). Similarly, other loxodromic isometries come in one-parameter families determined by the unique geodesic that they fix, and we refer to the vector fields generating these familes as \textit{loxodromic Killing fields}. We call the unique geodesic to which a loxodromic Killing field $L$ is tangent the \textit{axis} of $L$. Since the isometry group of $\Mm$ is transitive on the set of geodesics, for any geodesic $\gamma\subseteq \Mm$, there exists a loxodromic Killing field with axis $\gamma$.
\begin{nota}
    Let $e$ be an edge of of geodesic triangle $T\subseteq\Mm$. Then, in any model, we often denote a Loxodromic vector field tangent to $e$ by $L_e$.
\end{nota}
\indent In the Poincar\'e disk $D_{-1}$, it is easy to see from the metric (\ref{poincaremetric}) that the Euclidean rotation about the origin by any angle defines an isometry $D_{-1}\to D_{-1}$ that is called an \textit{elliptic} isometry. An infinitesimal generator of these isometries is given by \[R:=-y\partial_x+x\partial_y.\] Moreover, any conjugate of this family of isometries is another family of elliptic isometries, and we refer to a generator of one of these families as an \textit{elliptic Killing field}. As with $R$, an elliptic Killing field has a unique point at which it vanishes. We call this point the \textit{center} of the Killing field. By the transitivity of the isometry group of $\Mm$, for any $p\in \Mm$, there exists an elliptic Killing field centered at $p$.

\subsection{Spherical Killing fields}
On $2$-spheres, the set of all Killing fields has a simpler description since there is only one type: rotational. For any fixed line in $\Rbb^3$ that passes through the origin, there is a unique Killing field (up to scaling) on $S^2$ with zeros at the intersection of the sphere with this line.\\
\indent Though there is only one type of Killing field on a sphere, for ease of exposition, we apply the terminology for hyperbolic Killing fields to spherical Killing fields. Namely, given any geodesic on a sphere, there is a unique (up to scaling) Killing field tangent to this geodesic. We call this Killing field loxodromic, and we call this geodesic its axis. Moreover, for $\kappa>0$, for any point $p\in \Mk$, there is a unique (up to scaling) Killing field vanishing at $p$. We call this Killing field elliptic with center $p$. 

\subsection{Properties of Killing fields} We will make use of two key properties of Killing fields on $\Mk$. The first is clear from construction: An elliptic vector field is orthogonal to every geodesic passing through its center. The second follows from a well-known general fact about Killing fields. 
\begin{lem}\label{nowhereortho}
    Let $L$ be a loxodromic Killing field on $\Mk$ with axis $\gamma$. Let $\eta\subseteq \Mk$ be a geodesic that does not intersect $\gamma$ orthogonally. Then $L$ is nowhere orthogonal to $\eta$. Let $R$ be an elliptic Killing field on $\Mk$ with center $p$. If $\eta$ is a geodesic not containing $p$, then $L$ is nowhere orthogonal to $\eta$. 
\end{lem}
\begin{proof}
    Recall (see \cite{lee} Exercise 6-24) that, on any Riemannian manifold, a Killing field is either everywhere or nowhere orthogonal to a given geodesic. Suppose that $L$ is a loxodromic Killing field on $\Mk$ with axis $\gamma$. If $\eta$ is not orthogonal to $\gamma$, then there exists an equidistant curve to $\gamma$ intersecting $\eta$ non-orthogonally. Since $L$ is tangent to these equidistant curves, there is a point on $\eta$ at which $L$ is not orthogonal to $\eta$. Thus, $L$ is nowhere orthogonal to $\eta$.\\
    \indent Let $R$ be a rotational vector field centered at $p\in \Mk$, and let $\eta$ be a geodesic not containing $p$. There exists a point $q\in\eta$ such that the geodesic joining $p$ to $q$ is not tangent to $\eta$. Since $R$ is orthogonal to this geodesic, it is nowhere orthogonal to $\eta$. 
\end{proof}
\section{Nodal sets of low-energy eigenfunctions and their derivatives}\label{nodalderivsection}
We now apply the Killing fields and their properties introduced in the last section to first mixed and second Neumann eigenfunctions on a geodesic polygon $P$ in $\Mk$. Fix a mixed eigenvalue problem by defining $D,N\subseteq\partial P$. If $D\neq \emptyset$, we let $v$ be a first mixed eigenfunction and let $\lambda$ be the corresponding first mixed eigenvalue. If $D=\emptyset$, we let $u$ be a second Neumann eigenfunction and let $\mu$ be the corresponding second Neumann eigenvalue.

\begin{lem}\label{neumannloop}
    Let $T$ be a geodesic triangle in $\Mk$. Let $e\subseteq N$ be the closure of an edge of $T$, and let $X$ be a Killing field tangent to $e$. 
    \begin{enumerate}
        \item If $D\neq\emptyset$, let $\phi=Xv$. Then $\Zcal(\phi)$ cannot contain a loop or an arc with two endpoints in $e$ unless $\phi$ vanishes identically on $\partial T\setminus e$.
        \item If $D=\emptyset$ and $\kappa<0$, let $\phi$ equal either $u$ or $Xu$. Then $\Zcal(\phi)$ cannot contain a loop or an arc with two endpoints in $e$ unless $\phi$ vanishes identically on $\partial T\setminus e$.
    \end{enumerate}
\end{lem}
\begin{proof}
    For (1), suppose to the contrary that $\Zcal(\phi)$ contains such a loop or arc and that $\Zcal(\phi)$ does not contain $\partial T\setminus e$. Let $\Omega$ be the topological disk bounded by this arc and, if necessary, a subset of the edge $e$. Because $\Zcal(\phi)$ does not contain $\partial T\setminus e$, $T\setminus \Omega$ has non-empty interior. By Proposition \ref{regularity}, the function $\psi:=\phi\cdot\chi_{\Omega}$ (where $\chi_{\Omega}$ is the indicator function for $\Omega$) is a valid test function for the mixed problem defined by $D$ and $N$. Moreover, since $X$ is tangent to $e$, $Xv$ satisfies Neumann boundary conditions on $e\cap \partial \Omega$. Integrating by parts, we have \[\int_T|\nabla\psi|^2=\lambda\int_T|\psi|^2+\int_{\partial T}\psi\partial_{\nu}\overline{\psi}=\lambda\int_T|\psi|^2,\] so we see that $\psi$ is a first mixed eigenfunction of $T$. Because $\psi$ vanishes on an open set, this contradicts the fact that eigenfunctions are real-analytic.\\
    \indent For (2), let $\phi$ equal either $u$ or $Xu$ and suppose that the negation holds. Let $\lambda_1$ be the first mixed eigenvalue for the problem with Neumann conditions on $e$ and Dirichlet conditions on $\partial T\setminus e$. If $\kappa<0$, then Lemma \ref{mixedineq} implies that $\mu\leq \lambda_1$, but the same computation as in the last paragraph shows that $\lambda_1<\mu$, contradicting Lemma \ref{mixedineq}.
\end{proof}

The next three corollaries mirror Theorem 5.2 and Corollaries 5.3-5.5 of \cite{judgemondal}.

\begin{coro}\label{nodalsimplearc}
    Let $u$ be a second Neumann eigenfunction on a geodesic triangle $T\subseteq \Mk$ with $\kappa<0$. Then $\Zcal(u)$ is a simple arc with distinct endpoints in $\partial T$, and these endpoints do not lie in the closure of the same edge of $T$. In particular, $u$ cannot vanish at two distinct vertices of $T$. 
\end{coro}
\begin{proof}
    Since $u$ is orthogonal to the constant functions, it must change signs in $T$, so $\Zcal(u)\cap T$ contains an arc. By the Hopf lemma \cite{hopf}, $u$ does not vanish identically on any edge of $T$. By Lemma \ref{neumannloop}, this arc cannot be a loop, and its endpoints cannot be located in the closure of the same edge of $T$. By Courant's nodal domains theorem, there can be only one such arc.
\end{proof}

\begin{coro}\label{cantbothvanish}
    If $u$ is a second Neumann eigenfunction on a geodesic triangle $T\subseteq \Mk$ with $\kappa<0$ and $v$ is a vertex of $T$, then in expansion (\ref{neumannexpansion}) at $v$ for $u$, the coefficients $a_0$ and $a_1$ cannot both vanish.
\end{coro}
\begin{proof}
    If both of these coefficients were to vanish, then the last statement in Lemma \ref{arcsnearvertex} would give a contradiction to Corollary \ref{nodalsimplearc}.
\end{proof}

\begin{coro}\label{atmosttwo}
    The second Neumann eigenspace of a negative constant-curvature geodesic triangle $T$ has dimension at most two.
\end{coro}
\begin{proof}
    Call this eigenspace $E$. For a fixed vertex $v$ of $T$, the coefficient map \[E\ni u\mapsto (a_0(u,v),a_1(u,v))\in \Rbb^2\] (with the $a_i$ as in (\ref{neumannexpansion})) is linear and injective by Corollary \ref{cantbothvanish}.
\end{proof}

\begin{lem}\label{coeffnonzero}
    If $u$ is a first mixed eigenfunction of a geodesic polygon $P\subseteq\Mk$ (for any $\kappa\in \Rbb$), then at each mixed (resp. Dirichlet) vertex of $P$, the lowest coefficient $b_0$ (resp. $c_1$) of expansion (\ref{mixedexpansion}) (resp. (\ref{dirichletexpansion})) is non-zero.
\end{lem}
\begin{proof}
    First mixed eigenfunctions do not vanish outside of their Dirichlet region. By Lemmas \ref{mixedarcsatvertex} and \ref{dirarcsatvertex}, if the lowest coefficient in the expansion were to vanish, then $u$ would vanish in $P$.
\end{proof}

The remaining results of this section determine the behavior of nodal sets of Killing fields applied to eigenfunctions near the vertices of a geodesic triangle. The computations are very similar to the ones performed to prove Lemma 2.2 of \cite{erratum}. We prove only the first of these three lemmas since the other two are similar to the first. 

\begin{lem}\label{derivativeneumannvertex}
    Let $u$ be a Laplace eigenfunction on a constant-curvature triangle $T\subseteq \Mk$ satisfying Neumann boundary conditions near an acute vertex $v$ of angle $\beta$. Suppose that $T$ is contained in the disk model of $\Mk$ with $v$ at the origin as in Section \ref{vertexsection}. Suppose that $X$ is a Killing field for $\Mk$ that does not vanish at $v$. Let $\alpha$ denote the angle between $X(v)$ and the positive $x$-axis. Then 
    \begin{enumerate}
        \item If $u(v)=0$ and $a_1(u,v)\neq 0$ (with $a_1$ as in Section \ref{vertexsection}), then $v$ is an endpoint of an arc in $\Zcal(Xu)\cap T$ if and only if the angle $\alpha\in[\beta-\pi/2,\pi/2]\cup[\beta+\pi/2,3\pi/2]$. If $v$ is an endpoint of such an arc, then it is an endpoint of exactly one such arc.
        \item If $u(v)\neq 0$, then $v$ is an endpoint of an arc of $\Zcal(Xu)$ if and only if $\alpha\in[\pi/2,\pi/2+\beta]\cup[3\pi/2,3\pi/2+\beta].$ If $v$ is an endpoint of such an arc, then it is an endpoint of exactly one such arc. 
    \end{enumerate}
\end{lem}
\begin{proof}
     We may, without loss of generality, rescale the vector field $X$ such that \[X(x,y)=\Big(\cos\alpha +O(\sqrt{x^2+y^2})\Big)\partial_x+\Big(\sin\alpha +O(\sqrt{x^2+y^2})\Big)\partial_y.\] In polar coordinates, the may be rewritten as \[X(r,\theta)=\Big(\cos(\alpha-\theta)+O(r)\Big)\partial_r+\Big(\frac1r\sin(\alpha-\theta)+O(1)\Big)\partial_{\theta}.\]
     Suppose that $v$ has interior angle $\beta<\pi/2$. Since $v$ is acute, $\nu=\pi/\beta>2$.\\
     \indent Suppose that $u(v)=0$ and $a_1(u,v)\neq 0$. Then, using the expansion (\ref{neumannexpansion}), \[u(r,\theta)=a_1r^{\nu}g_{\nu}(0)\cos(\nu\theta)+O(r^{\nu+2}).\] Hence, \[Xu(r,\theta)=\cos(\alpha-\theta+\nu\theta)\nu a_1r^{\nu-1}g_{\nu}(0)+O(r^{\nu}).\] Since $0\leq \theta\leq \beta$, $v$ is an endpoint of an an arc in $\Zcal(Xu)$ if and only if $\alpha\in[\beta-\pi/2,\pi/2]\cup[\beta+\pi/2,3\pi/2]$.\\
     \indent Now suppose that $u(v)\neq 0$, so $a_0(u,v)\neq 0$. Then since $\nu>2$, using expansion (\ref{neumannexpansion}), 
     \[u(r,\theta)=a_0g_0(0)+a_0g_0'(0)r^2+O(r^{\min\{4,\nu\}}),\] so \[Xu(r,\theta)=\cos(\alpha-\theta)\cdot 2a_0g_0'(0)r+O(r^{\min\{4,\nu\}-1}).\] Thus, $v$ is an endpoint of an arc in $\Zcal(Xu)$ if and only if $\alpha\in[\pi/2,\pi/2+\beta]\cup[3\pi/2,3\pi/2+\beta]$.
\end{proof}

\begin{lem}\label{derivativemixedvertex}
     Let $u$, $T$, $v$, $\beta$, $X$, and $\alpha$ be as in Lemma \ref{derivativeneumannvertex} except with $\beta\in(0,\pi)$ and with $u$ satisfying Dirichlet conditions on the positive $x$-axis and Neumann conditions on the other edge adjacent to $v$. Suppose $b_0(u,v)\neq 0$. Then 
     \begin{enumerate}
         \item If $v$ is acute, then $v$ is an endpoint of an arc in $\Zcal(Xu)$ if and only if $\alpha\in[\beta-\pi/2,0]\cup[\beta+\pi/2,\pi]$.
         \item If $v$ is non-acute, then $v$ is an endpoint of an arc in $\Zcal(Xu)$ if and only if $\alpha\in[0,\beta-\pi/2]\cup[\pi,\beta+\pi/2]$.
     \end{enumerate}
     In either case, if $v$ is an endpoint of of an arc in $\Zcal(Xu)$, then it is an endpoint of exactly one such arc. 
\end{lem}

\begin{lem}\label{derivativedirichletvertex}
    Let $u$, $T$, $v$, $\beta$, $X$, and $\alpha$ be as in Lemma \ref{derivativeneumannvertex} except with $\beta\in(0,\pi)$ and with $u$ satisfying Dirichlet conditions near $v$. If $c_1(u,v)\neq 0$, then $v$ is an endpoint of an arc in $\Zcal(Xu)$ if and only if $\alpha\in [\beta-\pi,0]\cup [\beta,\pi]$. If $v$ is an endpoint of such an arc, then it is the endpoint of exactly one arc. 
\end{lem}

\section{Proof of Theorem \ref{finite}}\label{finitesection}
In this section we prove Theorem \ref{finite}.
\begin{lem}\label{nodircp}
    Let $u$ be a first mixed eigenfunction on a geodesic polygon $P\subseteq \Mk$ (for any $\kappa\in \Rbb$) satisfying Dirichlet boundary conditions on $D\subseteq \partial P$. Then $u$ has no (non-vertex) critical points in $D$. 
\end{lem}
\begin{proof}
    Recall that first mixed eigenfunctions do not vanish inside $P$. By viewing the Laplace operator $\Dk$ as an elliptic operator acting on $P\subseteq \Rbb^2$, we may apply Hopf's lemma \cite{hopf} to see that the normal derivative of $u$ at smooth points of $D$ is never equal to zero. Thus, these points are not critical points of $u$.
\end{proof}

\begin{lem}\label{noaccummixed}
    Let $u$ be a first mixed eigenfunction on a geodesic polygon $P\subseteq \Mk$ (for any $\kappa\in \Rbb$). Then no mixed or Dirichlet vertex of $P$ is an accumulation point of $\crit(u)$.
\end{lem}
\begin{proof}
    This is an immediate consequence of Lemmas \ref{mixedarcsatvertex}, \ref{dirarcsatvertex}, and \ref{coeffnonzero}.
\end{proof}

Before proving Theorem \ref{finite}, we reduce the problem to showing that no edge is contained in the critical set of a first mixed or second Neumann eigenfunction.

\begin{prop}\label{nocritarcs}
    Let $P\subseteq \Mk$ be a simply connected geodesic polygon for some $\kappa\in \Rbb$. Let $D$ be a connected union of edges of $P$. Let $u$ be either a second Neumann eigenfunction for $P$ or a first mixed eigenfunction for $P$ with Dirichlet conditions on $D$ and Neumann conditions on $\partial P\setminus D$. Then $\crit(u)$ cannot contain an arc intersecting (the interior of) $P$. 
\end{prop}
\begin{proof}
    Suppose that such an arc exists. Let $\gamma$ a connected component of this arc contained in (the interior of) $P$. By Lemmas \ref{noaccumatvertex}, \ref{mixedarcsatvertex}, and \ref{dirarcsatvertex}, $\gamma$ cannot have an endpoint at a vertex of $P$. If $u$ is a first mixed eigenfunction, then $\gamma$ also cannot have an endpoint in $D$ by Lemma \ref{nodircp}. If $u$ is a second Neumann eigenfunction, then $\gamma$ cannot intersect $\Zcal(u)$ or else $u$ would vanish on $\gamma$ and contradict Lemma \ref{nodalcritical}. Thus, in either case, $\gamma$ is either a loop, or $\gamma\cup (\partial P\setminus D)$ contains a loop since $P$ is simply connected and $D$ is connected. This loop bounds some topological disk $\Omega$. Then $u$ does not change signs in $\Omega$ and satisfies Neumann boundary conditions in $\Omega$. However, a non-constant Neumann eigenfunction must change signs in order to be orthogonal to constant functions, a contradiction.
\end{proof}

\begin{proof}[Proof of Theorem \ref{finite}]
    We first restrict to the negative curvature case. Because the metrics on $\Mk$ and $M_{\kappa'}$ differ by a scalar multiple for $\kappa,\kappa'<0$, it suffices to prove the theorem in the special case that $P\subseteq \Mm$. Suppose that $u$ is either a first mixed or second Neumann eigenfunction on $P$. If $u$ is a first mixed eigenfunction, suppose that the corresponding Dirichlet region is a connected union of edges.\\
    \indent Suppose toward a contradiction that $\crit(u)$ is infinite. By Lemma \ref{noaccummixed}, no mixed or Dirichlet vertex of $P$ is an accumulation point of $\crit(u)$. If $\crit(u)$ has an interior accumulation point, then the real-analyticity of $|\Gk u|_{\kappa}^2$ would imply that $\crit(u)$ contains an arc intersecting $P$, which would contradict Proposition \ref{nocritarcs}. By Lemma \ref{noaccumatvertex}, if a Neumann vertex is an accumulation point of $\crit(u)$, then this vertex has angle equal to either $\pi/2$ or $3\pi/2$, and near this vertex, the critical set is contained in exactly one of its adjacent edges. By analyticity, this edge is entirely contained in $\crit(u)$, so both vertices adjacent to this edge have interior angle in $\{\pi/2,3\pi/2\}$, and both edges adjacent to this edge are contained in the Neumann region.\\
    \indent By applying a hyperbolic isometry, we may suppose that the edge $e$ of $P$ that is contained in $\crit(u)$ is contained in the $y$-axis in the upper half-plane model of $\Mm$. Let $X=x\partial_x+y\partial_y$ denote the loxodromic Killing field that is tangent to the $y$-axis. Then $Xu$ vanishes identically on $e$, implying that $\partial_y (Xu)\equiv 0$ on $e$. Since $Xu$ satisfies Neumann boundary conditions on $e$, we also have $\partial_x(Xu)\equiv 0$ on $e$. Thus, $e$ is contained in the nodal critical set of $Xu$. Since $Xu$ satisfies the eigenfunction equation, this implies that $Xu\equiv 0$ on $P$ by Lemma \ref{nodalcritical}. Working in polar coordinates on the upper half-plane, this implies that $u(r,\theta)$ depends only on $\theta$.\footnote{Recall from Section \ref{killingfields} that the vector field $X$ is tangent to curves in $\Mm$ equidistant to the $y$-axis.} Viewed as a function of $\theta$, $u$ then is a solution to the ordinary differential equation \[\mu u=-\sin^2(\theta)u''\] where $\mu$ is the eigenvalue of $u$.\\
    \indent We claim that $P$ cannot have a Neumann edge that is not contained in either the $y$-axis or a geodesic represented by a semi-circle centered at the origin. Indeed, suppose to the contrary that $P$ has such an edge $e'$. Then the outward normal vector to $e'$ is linearly independent to $\partial_r$ everywhere on $e'$. Since $u$ satisfies Neumann conditions on $e'$ and $\partial_r u$ vanishes on $e'$, this implies that $u$ is a constant function, a contradiction.\\
    \indent Since $P$ cannot have such a Neumann edge, we have a contradiction in the case that $u$ is a second Neumann eigenfunction since no such polygon $P$ exists.\\
    \indent If $u$ is a first mixed eigenfunction, then its Dirichlet region is contained in a disjoint union of curves equidistant to the $y$-axis. These curves are not geodesics, so we have another contradiction.\\
    \indent We now prove the hemispherical case, which is not very different from the case of negative curvature. By the same argument as above, if $P$ is a simply connected spherical polygon and $u$ is a first mixed (with $D$ connected) or second Neumann eigenfunction with infinitely many critical points, then $P$ has an edge $e\subseteq \crit(u)$ whose adjacent vertices are Neumann vertices with interior angle equal to either $\pi/2$ or $3\pi/2$. If $L_e$ is a Killing field tangent to $e$, then $L_eu$ vanishes identically in $P$, so $u$ is a function only of the distance to $e$ and is therefore constant on latitudes equidistant to the geodesic containing $e$. Again, $P$ can have no edge that is not contained in the geodesic containing $e$ or a geodesic orthogonal to this geodesic. Thus, if $u$ is a first mixed eigenfunction, then $D$ is contained in a finite union of latitudes a positive distance to the geodesic containing $e$, so $D$ is not a finite union of geodesics, giving a contradiction. If $u$ is a second Neumann eigenfunction, then the only possibilities are that $P$ is an isosceles triangle with two right-angled vertices or that $P$ is the union of multiple such triangles whose closures intersect only in the geodesic containing $e$. In case of the latter, the intersection of $P$ with the geodesic containing $e$ is contained in $\crit(u)$, contradicting Proposition \ref{nocritarcs}.
\end{proof}

\section{Paths of triangles and Riemannian metrics: continuity of the spectrum and stability of critical points}\label{perttheory}

The remaining main results of the paper are proved by considering paths of triangles with accompanying paths of first mixed and second Neumann eigenfunctions that limit to a triangle for which we already know useful information. In this section, we define these paths and derive some properties of the behavior of critical points of eigenfunctions on domains in these paths.\\
\indent For $\kappa\in \Rbb\setminus \{0\}$, fix a geodesic triangle $T_{\kappa}\subseteq \Mk$. A convenient choice of coordinates for our path is given by the Klein model, whose metric is given by (\ref{kleinmetric}). Viewing $T_{\kappa}$ as being a subset $T\subseteq \Rbb^2$ endowed with this metric, if $\kappa<0$ (resp. $\kappa>0$), then for $\tau\in [\kappa,0]$ (resp. $\tau\in[0,\kappa]$), define geodesic triangles $T_{\tau}$ by endowing the Klein metric (\ref{kleinmetric}) with curvature $\tau$ onto the set $T\subseteq \Rbb^2$. We further assume that 
\begin{itemize}
    \item In the Neumann case, the largest vertex of $\Tk$ is located at the origin.
    \item In the mixed case with one Dirichlet edge, the Neumann vertex is located at the origin.
\end{itemize}
Since the Klein metric is conformal at the origin, this will ensure that the hypotheses in Theorems \ref{mainthm} and \ref{mixed} are satisfied for every triangle along the path. For the mixed case with two Dirichlet edges, we do not make use of a path of metrics, so for the remainder of this section, every reference to the mixed problem refers to the problem with one Dirichlet edge. To avoid notational burdens, we re-parameterize this path by $[0,1]\ni t\mapsto T_t$ such that $T_0$ has curvature equal to $0$ and $T_1=T_{\kappa}$. Say that triangle $T_t$ has curvature $\kappa(t)$.\\
\indent Fix either mixed or Neumann boundary conditions on $T$. By standard results in linear perturbation theory (see, e.g., \cite{kato}), there exist countably infinitely many real-analytic maps $\mu_j:[0,1]\to [0,\infty)$ such that $\mu_j(t)$ is an eigenvalue of $-\Delta_{\kappa(t)}$ on $T_{t}$ and such that every such eigenvalue is the value of one of these functions. Note however that these paths of eigenvalues may cross, so we may have $\mu_j(t)<\mu_k(t)$ with $k<j$ for some $t$, $j$, and $k$. There also exist correspondng real-analytic paths $u_j:[0,1]\to H^1(T)$ of Laplace eigenfunctions.\\
\indent If we are working with mixed boundary conditions, then the first mixed eigenvalue is always a simple eigenvalue, so we may relabel the $\mu_j$ and $u_j$ such that $\mu_1(t)$ is a first mixed eigenfunction for $T_t$ with first mixed eigenfunction $u_1(t)$ for all $t$.\\
\indent In the Neumann case, a complication in the construction of our paths of triangles is that, \textit{a priori}, these paths of eigenvalues may cross over each other (representing spectral multiplicities), complicating the construction of a continuous path of second Neumann eigenfunctions to accompany the path of triangles. The next two lemmas resolve this issue.
\begin{lem}\label{finitelymanysecond}
    Suppose that $T_1$ is not acute and that $\kappa<0$. Then for all but at most finitely many $t\in[0,1]$, the triangle $T_t$ has simple second Neumann eigenvalue. 
\end{lem}
\begin{proof}
    Since $T_1$ is not acute, $T_0$ is also not acute. By the simplicity result of Atar and Burdzy \cite{atarburdzy}, the triangle $T_0$ has simple second Neumann eigenvalue. By the real-analyticity of the eigenvalue branches, if there were infinitely many $t$ such that $T_t$ has a multiple second Neumann eigenvalue, then there would exist $j$ such that the eigenvalue $\mu_j(t)$ has multiplicity at least $2$ for all $t\in [0,1]$. Since the second Neumann eigenvalue of $T_0$ is not equal to $\mu_j(0)$, there must exist $t_0$ in $(0,1)$ such that $\mu_j(t_0)=\mu_i(t_0)$ for some $i\neq j$ and such that $\mu_j(t_0)$ is the second Neumann eigenvalue of $T_{t_0}$. Thus $T_{t_0}$ has a second Neumann eigenvalue of multiplicity equal to at least $3$, contradicting Corollary \ref{atmosttwo}.
\end{proof}
That there are only finitely many points on the path such that there is a spectral multiplicity, we may find a continuous path of second Neumann eigenfunctions. This is similar to the result of Section 12 of \cite{judgemondal}.
\begin{coro}\label{ctspath}
    There exists a reparameterization $[0,1]\ni s\mapsto T_s$ of $t\mapsto T_t$ and a function $[0,1]\ni s\mapsto \kappa(s)$ such that $T_s\subseteq M_{\kappa(s)}$ as well as a continuous path \[[0,1]\ni s\mapsto u_s\in H^1(T)\] of second Neumann eigenfunctions for $T_s$. 
\end{coro}
\begin{proof}
    By Lemma \ref{finitelymanysecond}, the set of $t$ for which $T_t$ has a simple second Neumann eigenvalue is a finite union of open intervals, and on the closure of each of these intervals, standard perturbation theory gives the continuous paths we desire.\\
    \indent Suppose that for $t_0\in (0,1]$, triangle $T_{t_0}$ has a multiple second Neumann eigenvalue. By the previous paragraph, $T_{t_0}$ is the endpoint of two distinct open intervals on which our path of eigenfunctions is already defined. Within the second Neumann eigenspace, there exists a continuous path of (non-zero) eigenfunctions connecting these two eigenfunctions. Concatenating all of these paths together gives the desired path. 
\end{proof}

As allowed by Corollary \ref{ctspath}, for the remainder of the paper, we suppose without loss of generality that the path of triangles $T_t$ has an accompanying path $u_t$ of second Neumann (or first mixed) eigenfunctions. For the remainder of this section, we study the behavior of these paths of eigenfunctions. 

\begin{defn}
    Suppose that eigenfunction $u_t$ has a non-vertex critical point $p\in \overline{T}_t$. Let $T\subseteq \Rbb^2$ be the set in $\Rbb^2$ representing $T_t$. We will say that $p$ is \textit{stable under perturbation} if for any neighborhood $U$ of $p$ in $\overline{T}$, there exists $\epsilon>0$ such that if $|t-s|<\epsilon$, then $u_t$ has a critical point in $U$. 
\end{defn}

The first result on the stability of critical points is standard (see Lemma 10.1 \cite{judgemondal}), and we omit the proof. Recall that a critical point $p$ of a smooth function $u$ is called \textit{nondegenerate} if the Hessian of $u$ has non-vanishing determinant at $p$. Otherwise, we call $p$ a \textit{degenerate} critical point.
\begin{lem}\label{nondegenstable}
    Nondegenerate critical points in $T_t$ are stable under perturbation. 
\end{lem}

\begin{lem}\label{stablecp}
    Let $X_t$ denote an elliptic or loxodromic Killing field on $T_t$. If $u_t$ has a non-vertex critical point $p\in \partial T_t$ such that $p$ is an endpoint of exactly one arc in $\Zcal(X_tu_t)\cap T_t$ and such that the edge containing $p$ is not contained in $\Zcal(X_tu_t)$, then $p$ is stable under perturbation.
\end{lem}
\begin{proof}
    By the constructions in Section \ref{killingfields}, there exists a real-analytic family of Killing fields $(t-\epsilon,t+\epsilon)\ni s\mapsto X_s$ that limit to $X_t$ as $s\to t$. Extending $u_t$ to a neighborhood of $p$ via reflection over the edge containing $p$, Lemma \ref{nodalcritical} shows that, near $p$, the set $\Zcal(X_t u_t)$ is a simple arc that intersects both $T$ and the complement of the closure of $T$. By the continuity in $H^1$ of $s\mapsto u_s$ and by extending the eigenfunction via reflection to a neighborhood of $p$ in $\Rbb^2$, we see that there is a neighborhood $U$ of $p$ such that for $s$ near $t$, the zero level set $\Zcal(X_su_s)\cap U$ is also a simple arc intersecting $\partial T_s$ transversely. Since the edge containing $p$ is not contained in $\Zcal(X_tu_t)$, for $s$ sufficiently close to $t$, $X_s$ is linearly independent to the normal vector to the edge of $T_s$ that $U$ intersects. Thus, this intersection is a critical point of $u_s$. 
\end{proof}

The remaining results of this section use the vertex expansions (\ref{neumannexpansion}) and (\ref{mixedexpansion}) to understand the behavior of critical points near vertices as we vary the eigenfunction $u_t$. By the continuity of $t\mapsto u_t$ in $H^1(T)$, for a fixed vertex $v$ of $T$, the coefficients $a_j(u_t,v)$ and $c_j(u_t,v)$ in (\ref{neumannexpansion}) and (\ref{mixedexpansion}), respectively, vary continuously with $t$ as well. Lemmas \ref{convergethenzero} and \ref{convtosamevertex} are proved nearly identically to Lemma 9.3 and Lemma 9.2, respectively, of \cite{judgemondal}, so we omit the proofs. Lemma \ref{sequencetomixed} is new and proved similarly to the other two results, but we prove it for completeness. 

\begin{lem}\label{convergethenzero}
   Let $v$ be a Neumann vertex of $T$. Suppose that there exists a sequence $t_n\in [0,1]$ converging to some $t_0$ such that each $u_{t_n}$ has a non-vertex critical point $p_n$ such that $p_n\to v$ in $T$ as $n\to\infty$. If $T_{t_0}$ has an acute angle at $v$, then the eigenfunction $u_{t_0}$ vanishes at $v$. If $T_{t_0}$ has an obtuse angle at $v$, then the coefficient $a_1(u_{t_0},v)$ in (\ref{neumannexpansion}) vanishes. 
\end{lem}

\begin{lem}\label{convtosamevertex}
    Suppose that $u_t$ is a second Neumann eigenfunction for each $t$. Suppose that there exists a sequence $t_n\in [0,1]$ converging to some $t_0$ such that each $u_{t_n}$ has two distinct non-vertex critical points $p_n$, each contained in distinct edges of $T_{t_n}$. Then $p_n$ and $q_n$ do not both converge to the same vertex of $T_{t_0}$.
\end{lem}

\begin{lem}\label{sequencetomixed}
    Suppose that $u_t$ is a first mixed eigenfunction for each $t$. Suppose that there exists a sequence $t_n\in [0,1]$ converging to some $t_0$ such that each $u_{t_n}$ has a non-vertex critical point $p_n$ contained in a Neumann edge. Then $p_n$ does not converge to a mixed vertex of $T_{t_0}$.
\end{lem}
\begin{proof}
    Suppose to the contrary that such a sequence exists. We show that if $p_n$ converges to a mixed vertex $v$ of $T_{t_0}$, then the coefficient $b_0(u_{t_0},v)$ of (\ref{mixedexpansion}) vanishes, contradicting Lemma \ref{coeffnonzero}.\\
    \indent We now work in the Poincar\'e disk model in order to use the expansion (\ref{mixedexpansion}) with the vertex at the origin as in Section \ref{vertexsection}. In this model, the angle of the vertex changes with $n$, so we denote this angle by $\beta_n$. Similarly, we define $\nu_n=\pi/\beta_n$. We abbreviate the coefficients by $b_0(t_n):=b_0(u_{t_n},v)$. Suppose that for each $n$, $p_n=(r_n,\beta_n)$ is a critical point of $u_n$ contained in the Neumann edge adjacent to $v$. Since these are critical points, the expansion (\ref{mixedexpansion}) gives 
    \[0=\partial_r u_{t_n}(p_n)=\frac12\nu_nb_0(t_n)r_n^{\nu_n/2-1}g_{\nu_n/2}(0)+O(r_n^{3\nu_n/2-1})+O(r_n^{\nu_n/2+1})\] as $n\to\infty$. Dividing both sides by $r_n^{\nu_n/2-1}$ gives \[0=\frac12\nu_nb_0(t_n)g_{\nu_n/2}(0)+O(r_n^{\nu_n})+O(r_n^2).\] Since $\nu_n$ is bounded below by a positive constant and $g_{\nu_n/2}(0)=1$, it follows that $b_0(t_n)\to0$ as $n\to\infty$. By the continuity of the coefficients along the path of triangles, we get $b_0(u_{t_0},v)=0$.
\end{proof}

\section{Critical points that imply the existence of other critical points}\label{implicationsection}

Continuing our study of the stability properties of critical points initiated in Section \ref{perttheory}, we now seek to understand when the existence of certain critical points imply the existence of other, often more stable, critical points. We do so primarily by studying the zero-level sets of Killing fields applied to eigenfunctions. Throughout the section, let $T\subseteq \Mk$ be an open geodesic triangle for any $\kappa\in \Rbb$. Unless otherwise stated, we will take $u$ to be a first mixed or second Neumann eigenfunction of $T$. When $u$ is a second Neumann eigenfunction, we suppose that $\kappa<0$.

\begin{lem}\label{arcsnearcp}
    Suppose that $p$ is a critical point of $u$ contained in a Neumann edge $e$. If $L_e$ is a Killing field tangent to $e$, then $p$ is the endpoint of at least one arc of $\Zcal(L_eu)\cap T$.
\end{lem}
\begin{proof}
    Note that $L_e$ is invariant under reflection about the geodesic containing $e$. Since we can extend $u$ analytically to be even with respect to this reflection, $L_eu$ is also even with respect to this reflection. Since $L_eu$ is a Laplace eigenfunction, Lemma \ref{nodalcritical} implies that $p$ is contained in an arc in $\Zcal(L_eu)$. By Theorem \ref{finite}, this arc does not lie in $e$. Since this arc is reflection invariant, it must intersect $T$. 
\end{proof}

\begin{lem}\label{degenimpliesnondegen}
    Suppose that $u$ is a second Neumann eigenfunction. If $u$ has a degenerate critical point in an edge $e$ of $T$, then $u$ has another critical point in a different edge of $T$, and this critical point is stable under perturbation.
\end{lem}
\begin{proof}
    Since $\kappa<0$, it suffices to consider the case $\kappa=-1$, and we work in the upper half-plane. Suppose that $e$ is contained in the $y$-axis and that $p=(0,1)$. Since $u$ extends to be even about the $y$-axis, the Taylor expansion of $u$ about $p$ is of the form 
    \begin{equation}\label{taylorexp}
        u(x,y)=u(p)+ax^2+b(y-1)^2+O(3),
    \end{equation}
    where if $d((x,y),(0,1))$ is the geodesic distance from $(x,y)$ to $(0,1)$, we denote \[O(k):=O(d((x,y),(0,1))^k).\] Then $p$ is a degenerate critical point if and only if either $a$ or $b$ equals $0$.\\
    \indent Suppose that $b=0$. As in Section \ref{killingfields}, $L_e=x\partial_x+y\partial_y$ is a Killing field tangent to $e$. We may rewrite this as $L_e=x\partial_x+[(y-1)+1]\partial_y$. Since $b=0$, this gives $L_eu(x,y)=O(2)$. Since $L_eu$ satisfies the eigenfunction equation, Lemma \ref{nodalcritical} implies that $\Zcal(Xu)$ has at least two arcs meeting transversely at $p$. By Theorem \ref{finite}, neither of these arcs is contained in the $y$-axis. Thus, at least two arcs in $\Zcal(Xu)\cap T$ emanate from $p$. These arcs may not intersect each other away from $p$ and may not have another endpoint in the closure of $e$ by Lemma \ref{neumannloop}. By Lemma \ref{derivativeneumannvertex}, at most one of these arcs can have the vertex opposite to $p$ as an endpoint. Thus, one of these arcs has a degree one vertex in the interior of an edge not equal to $e$. This endpoint is a critical point of $u$ that, by Lemma \ref{stablecp}, is stable under perturbation.\\
    \indent Now suppose that $a=0$. Let $R$ be an elliptic Killing field centered at a vertex $v$ adjacent to $e$. Near $p$, after rescaling if necessary, $R$ may be expressed as \[R=(1+O(1))\partial_x+O(1)\partial_y\] (note that in this nonstandard notation, $O(1)=O(d((x,y),(0,1))$). Then we have $Ru(x,y)=O(2)$. Once again, Lemma \ref{nodalcritical} implies that at least two arcs in $\Zcal(Ru)$ intersect transversely at $p$. Since $R$ is orthogonal to two edges of $T$, one of these arcs is contained in the $y$-axis, and the other edge adjacent to $v$ is also contained in $\Zcal(Ru)$. By Lemma \ref{neumannloop}, the arc in $\Zcal(Ru)\cap T$ with an endpoint at $p$ must have its other endpoint in the interior of the edge opposite to $v$. This point is a critical point that, by Lemma \ref{stablecp}, is stable under perturbation. 
 \end{proof}

\begin{lem}\label{Nintimpliesext}
    If $u$ is a second Neumann eigenfunction with an interior critical point, then $u$ has at least one critical point in the interior of each edge of $T$. 
\end{lem}
\begin{proof}
    This is similar to the last part of the proof of Lemma \ref{degenimpliesnondegen}, and we omit the proof. See also Lemma 7.1 of \cite{judgemondal}.
\end{proof}

\begin{lem}\label{mixedintimpliesext}
    Suppose that $u$ is a first mixed eigenfunction satisfying Dirichlet conditions on exactly one edge of $T$. If $u$ has an interior critical point, then $u$ has at least one critical point on each Neumann edge. 
\end{lem}
\begin{proof}
    Let $v$ be a mixed vertex of $T$. Let $R$ denote an elliptic vector field centered at $v$. Then $Ru$ vanishes identically on the Neumann edge adjacent to $v$ and does not vanish anywhere in the interior of the Dirichlet edge by Lemma \ref{nodircp}. By Lemma \ref{derivativemixedvertex}, the other mixed vertex of $T$ is not an endpoint of an arc in $\Zcal(Ru)$. If $u$ has an interior critical point, then by Lemma \ref{nodalcritical}, an arc in $\Zcal(Ru)$ intersects $T$. By Lemma \ref{neumannloop}, this arc is not a loop, and at most one endpoint of this arc is in the closure of the Neumann edge adjacent to $v$. Thus, this arc has an endpoint in the edge opposite to $v$, and by Lemma \ref{nowhereortho}, this endpoint is a critical point of $u$. A symmetric argument yields another critical point on the other Neumann edge of $T$.
\end{proof}

The following should be compared with Proposition 8.1 \cite{judgemondal}. Judge and Mondal prove a similar statement in the case that the triangle is acute, but their proof does not carry over to the obtuse case.

\begin{lem}\label{morse}
    Suppose that $u$ is a second Neumann eigenfunction and that $T$ is obtuse. If $u$ has exactly one non-vertex critical point, then either every vertex of $T$ is a local extremum of $u$, or $u$ vanishes at an acute vertex of $T$. 
\end{lem}
\begin{proof}
    Suppose that $u$ has exactly one critical point $p$ and does not vanish at either acute vertex of $T$. By Lemma \ref{vertexmax}, each acute vertex of $T$ is a local extremum of $T$. By Lemma \ref{Nintimpliesext}, $p$ lies in the interior of an edge $e$ of $T$, and by Lemma \ref{degenimpliesnondegen}, $p$ is non-degenerate. From the Taylor expansion (\ref{taylorexp}), one can see that $p$ is therefore a local extremum of the restriction $u|_e$ of $u$ to $e$.\\
    \indent Since $\partial T$ is homeomorphic to a circle and $u|_{\partial T}$ has finitely many critical points (by Theorem \ref{finite}), $u|_{\partial T}$ has an even number of local extrema. By the Neumann boundary condition, each local extremum of $u|_e$ is either a critical point of $u$ or a vertex of $T$. Since the only critical point of $u$ is a local extremum of $u|_e$, the obtuse vertex must also be a local extremum of $u|_e$. One can check from the expansion (\ref{neumannexpansion}) that the obtuse vertex is therefore a local extremum of $u$ itself.  
\end{proof}

\begin{lem}\label{zerocoeffthencp}
    Let $u$ be a second Neumann eigenfunction, and suppose that $v$ is a vertex of $T$. If $a_1(u,v)=0$ in expansion (\ref{neumannexpansion}), then $u$ has a critical point on the edge of $T$ opposite to $v$.
\end{lem}
\begin{proof}
    Let $m$ be the smallest positive integer such that $a_m(u,v)=0$. Then
    \[u(r,\theta)=a_0g_0(r^2)+a_mr^{m\nu}g_{m\nu}(0)\cos(m\nu\theta)+O(r^{\min\{(m+1)\nu,m\nu+2\}}).\] Using polar coordinates in the Poincar\'e disk, $R:=\partial_{\theta}$ is an elliptic Killing field centered at the origin. Using the Neumann boundary conditions, the edges of $T$ adjacent to $v$ are contained in $\Zcal(Ru)$. Moreover, 
    \[Ru(r,\theta)=-a_mm\nu r^{m\nu}g_{m\nu}(0)\sin(m\nu\theta)+O(r^{\min\{(m+1)\nu,m\nu+2\}}).\] Since $m\geq 2$, $v$ is an endpoint of an arc in $\Zcal(Ru)\cap T$. By Lemma \ref{neumannloop}, the other endpoint of this arc is not contained in the closure of either edge adjacent to $v$. Thus, the other endpoint of this arc lies in the interior of the edge opposite to $v$. By Lemma \ref{nowhereortho}, this endpoint is thus a critical point of $u$. 
\end{proof}

\begin{lem}\label{zeroatacute}
    Suppose that $u$ is a second Neumann eigenfunction and that $T$ is not acute. If $u$ vanishes at an acute vertex $v$ of $T$, then $u$ has at least one non-vertex critical point.
\end{lem}
\begin{proof}
    As in the proof of Lemma \ref{morse}, the restriction $u|_{\partial T}$ has an even number of local extrema. Since $\overline{T}$ is compact and $u$ is continuous, this number is at least two. If $u$ vanishes at an acute vertex, then either $u$ has a non-vertex critical point or both of the other vertices of $T$ are local extrema of $u$. In the first case we are done.\\
    \indent If $T$ is obtuse and the obtuse vertex $v'$ of $T$ is a local extremum of $u$, then by Lemma \ref{vertexmax}, we have $a_1(u,v')=0$, so Lemma \ref{zerocoeffthencp} implies that $u$ has a non-vertex critical point.\\
    \indent Suppose that $T$ is a right triangle. Let $L$ be a loxodromic Killing field whose axis contains the leg of $T$ adjacent to $v$. Then $L$ is orthogonal to the other leg of $T$. Thus, the leg opposite to $v$ is contained in $\Zcal(Lu)$, and $Lu$ satisfies Neumann conditions on the leg adjacent to $v$. By Lemma \ref{derivativeneumannvertex}, $v$ is an endpoint of an arc in $\Zcal(Lu)\cap T$. By Lemma \ref{neumannloop}, the other endpoint of this arc must be in the interior of the hypotenuse of $T$, and by Lemma \ref{nowhereortho}, this endpoint is a critical point of $u$. 
\end{proof}

\begin{lem}\label{rightifanythenhyp}
    If $u$ is a second Neumann eigenfunction and $T$ is a right triangle, and $u$ has a critical point, then $u$ has at least one critical point on the hypotenuse of $T$.
\end{lem}
\begin{proof}
    By Lemma \ref{Nintimpliesext}, if $u$ has an interior critical point, then we are done. Suppose that $u$ has a critical point in a leg of $T$. Let $L$ be a loxodromic Killing field whose axis contains this leg. Then $Lu$ satisfies Neumann conditions on this leg and Dirichlet conditions on the other leg of $T$. By Lemma \ref{arcsnearcp}, the critical point is an endpoint of an arc in $\Zcal(Lu)$. By Lemma \ref{neumannloop}, the other endpoint of this arc must be in the interior of the hypotenuse. By Lemma \ref{nowhereortho}, this endpoint is a critical point of $u$. 
\end{proof}

Combining Lemmas \ref{zeroatacute} and \ref{rightifanythenhyp}, we get

\begin{coro}\label{rightexactlyone}
    Let $T$ be a right triangle with a second Neumann eigenfunction $u$. If $u$ vanishes at an acute vertex $v$ of $T$, then $u$ has a critical point on the hypotenuse of $T$.
\end{coro}

\begin{lem}\label{twoimpliestwoedges}
    Suppose that $u$ is a second Neumann eigenfunction. If $u$ has at least two critical points, then $u$ has critical points in at least two distinct edges of $T$.
\end{lem}
\begin{proof}
    If $u$ has an interior critical point, then we are done by Lemma \ref{Nintimpliesext}. Otherwise, suppose that $u$ has two critical points in the same edge $e$ of $T$. Let $L$ be a loxodromic Killing field whose axis contains $e$. By Lemma \ref{arcsnearcp}, each critical point in $e$ is an endpoint of an arc in $\Zcal(Lu)$. By Lemma \ref{neumannloop}, these arcs cannot intersect each other or themselves. Thus, at least one of these arcs has an endpoint in the interior of another edge of $T$. By Lemma \ref{nowhereortho}, this endpoint is a critical point of $u$. 
\end{proof}

\section{Proof of Theorem \ref{mixed}}\label{mixedsection}
We now prove Theorem \ref{mixed}, beginning with the case where the Dirichlet region equals two edges of $T$. The proof in this case applies equally as well for positive or negative curvature triangles, so we give the proof only for negatively curved triangles. By rescaling the metric, it suffices to prove the theorem for hyperbolic triangles. Let $T\subseteq \Mm$ be a geodesic triangle, and let $D\subseteq T$ equal the union of two edges of $T$ such that each mixed vertex has interior angle at most $\pi/2$. Let $u$ be a first mixed eigenfunction. In this case, the proof is nearly identical to the Euclidean case (see Proposition 7.1 of \cite{me}). 
\begin{proof}[Proof of Theorem \ref{mixed}, Part 1]
    \indent Let $L$ be the loxodromic vector field whose axis contains the shortest geodesic segment $\gamma$ joining the Dirichlet vertex to the geodesic containing the Neumann edge. This axis is orthogonal to the geodesic containing the Neumann edge, so the Neumann edge is contained in $\Zcal(Lu)$. Assume for a moment that both mixed vertices are acute. Then $T$ contains the segment $\gamma$, and it follows from Lemma \ref{derivativedirichletvertex} that no arc in $\Zcal(Lu)$ has an endpoint at the Dirichlet vertex. By Lemma \ref{derivativemixedvertex}, the only arc in $\Zcal(Lu)$ with an endpoint at a mixed vertex is the Neumann edge. By Lemma \ref{nodircp}, no arc in $\Zcal(Lu)$ has an endpoint in the interior of a Dirichlet edge. Thus, replacing $L$ with $-L$ if necessary, we have $Lu>0$ in $T$. If $T$ has a right-angled mixed vertex, then the leg of $T$ orthogonal to the Neumann edge is contained in $\Zcal(Lu)$. In this case, it follows from Lemmas \ref{nodircp} and \ref{neumannloop} that $\Zcal(Lu)$ does not intersect $T$. \\
   \indent Let $L_e$ be a loxodromic Killing field whose axis is the Neumann edge $e$. Each critical point of $u$ on the Neumann edge is an endoint of an arc in $\Zcal(L_eu)$ that intersects $T$, and by Lemma \ref{neumannloop}, these arcs do not intersect each other. By Lemma \ref{derivativedirichletvertex}, only one of these arcs can have an endpoint at the Dirichlet vertex. By Lemma \ref{derivativemixedvertex}, neither mixed vertex is an endpoint of an arc in $\Zcal(L_eu)$. By Lemma \ref{nodircp}, arcs in $\Zcal(L_eu)$ emanating from critical points in the Neumann edge do not have endpoints in the interiors of the Dirichlet edges. Thus, $u$ has at most one critical point in the Neumann edge. Since the restriction of $u$ to the Neumann edge must have a maximum and $u$ is Neumann on this edge, $u$ must have a critical point in the Neumann edge.
\end{proof}

We now prove the first half of the theorem. Let $T_{-1}\subseteq\Mm$ be a geodesic triangle with $D$ a single edge of $T_{-1}$ such that the vertex opposite to $D_1$ has angle less than $\pi/2$. As discussed in Section \ref{geometry}, we may represent $T_{-1}$ as a Euclidean triangle in the unit disk with the Klein metric. Moreover, we may suppose that the Neumann vertex of $T_{-1}$ is located at the origin. Let $T$ be the fixed subset of $\Rbb^2$ that represents $T_{-1}$ in this metric. For $\kappa\in [-1,0]$, let $\Tk$ be the geodesic triangle in $\Mk$ obtained by imposing the Klein metric (\ref{kleinmetric}) with curvature $\kappa$ onto $T$. Since the Klein metric is conformal at the origin, $\Tk$ has an acute-angled Neumann vertex for all $\kappa\in [-1,0]$. As discussed in Section \ref{perttheory}, the simplicity of the first mixed eigenvalue allows us to associate to this path of triangles an analytic path of first mixed eigenfunctions $u_{\kappa}\geq0$ and eigenvalues $\lambda_{\kappa}$. In what follows, let $e$ and $e'$ denote the Neumann edges of $T$, and let $e''$ denote the Dirichlet edge. \\
\indent Our proof of the first half of Theorem \ref{mixed} relies on the fact that the Euclidean version is already known to hold by Proposition 7.2 of \cite{me}. We are currently unable to give a direct proof as in \cite{me} because the hyperbolic version of the boundary integral formula in Lemma 4.2 of \cite{me} does not take such a simple form as in the Euclidean case. Instead, we prove that critical points on first mixed eigenfunctions of hyperbolic triangles have enough stability with respect to $\kappa$ such that the Euclidean case implies the hyperbolic case of the theorem.

\begin{prop}\label{equaledges}
   Suppose that $T$ has an acute Neumann vertex, and suppose that for some $\kappa\in[-1,0]$, $u_{\kappa}$ has a critical point in a Neumann edge. Then $u_{\kappa}$ has a critical point in a Neumann edge that is stable under perturbation.  
\end{prop}
\begin{proof}
    Suppose that $e$ is a Neumann edge of $T_{\kappa}$ containing a critical point of $u_{\kappa}$. Then by Lemma \ref{arcsnearcp}, each critical point on $e$ is an endpoint of at least one arc in $\Zcal (L_eu)$. If $e$ contains a critical point that is not a local extremum of $u|_e$, then by the Taylor expansion (\ref{taylorexp}), this critical point is the endpoint of at least two arcs in $\Zcal(L_eu)$. Otherwise, since the Neumann vertex is a local extremum of $u_{\kappa}$ and $u_{\kappa}$ vanishes at the other vertex of $e$, there are at least two critical points on $e$. Thus, in either case, there exist at least two arcs in $\Zcal(L_eu)\cap T$ with endpoints in $e$.\\
    \indent By Lemma \ref{neumannloop}, none of these arcs have another endpoint in the closure of $e$. By \ref{derivativeneumannvertex}, at most one of these arcs has an endpoint at the mixed vertex opposite to $e$. Thus, at least one of these arcs has a degree one endpoint in the interior of the other Neumann edge. By Lemma \ref{stablecp}, this endpoint is a critical point of $u_{\kappa}$ that is stable under perturbation. 
\end{proof}

\begin{proof}[Proof of Theorem \ref{mixed}, Part 2]
    As mentioned above, we begin by showing that in the case where $T$ has a single Dirichlet edge and Neumann vertex of angle less than $\pi/2$, the critical set is empty. By Lemmas \ref{nodircp} and \ref{mixedintimpliesext}, it suffices to prove that $u_{-1}$ has no critical points in the Neumann edges of $T_{-1}$. Let $S$ be the set of $\kappa\in [-1,0]$ such that $u_{\kappa}$ has a critical point on either $e$ or $e'$. We will show that $S$ is both open and closed in $[-1,0]$. Since $0\notin S$ by Proposition 7.2 of \cite{me}, this will imply that $S=\emptyset$.\\
    \indent The second statement in Proposition \ref{equaledges} implies that $S$ is open. To see that $S$ is closed, let $\{\kappa_n\}$ be a sequence in $S$ that converges to some $\kappa\in [-1,0]$. By re-labeling the edges and passing to a subsequence if necessary, we may suppose that, for each $n$, $u_{\kappa_n}$ has a critical point $p_n$ in $e$ that converges to some point $p$ in the closure of $e$. By Lemma \ref{convergethenzero}, if $p_n$ converges to the Neumann vertex, then $u_{\kappa}$ vanishes at the Neumann vertex, contradicting that $u_{\kappa}>0$ away from the Dirichlet edge. By Lemma \ref{sequencetomixed}, $p_n$ does not converge to a mixed vertex. Thus, $p$ is a point in the interior of $e$. By continuity, $p$ is a critical point of $u_{\kappa}$, so $\kappa\in S$. We have now established that $u_{-1}$ has no non-vertex critical points. Since the Neumann vertex is a local maximum of $u_{-1}$ by Lemma \ref{vertexmax}, it is also the global maximum.\\
    \indent Now let $L$ be a loxodromic Killing field perpendicular to the Dirichlet edge. Then $Lu_{-1}$ cannot vanish on $e''$ by Lemma \ref{nodircp}. Since $L$ is linearly independent to the outward normal derivative on $e$ and $e'$ and $u_{-1}$ has no critical points on these edges, $Lu_{-1}$ does not vanish on these edges. By Lemma \ref{derivativemixedvertex}, no mixed vertex is an endpoint of an arc in $\Zcal(Lu_{-1})$. By Lemma \ref{neumannloop}, therefore, $Lu$ does not vanish in $T$.\\
    \indent Finally, consider the case in which $T$ has a right angled Neumann vertex. Let $T'$ be the union of $T$ with its reflection over one of its Neumann edges. The extension of a first mixed eigenfunction of $T$ to be an even function on $T'$ yields a first mixed eigenfunction of $T'$ with Dirichlet conditions on two edges. Since the mixed vertices of $T'$ are both acute, the result follows from Part 1 of the proof of the theorem. 
\end{proof}

\section{Proof of Theorem \ref{mainthm}}\label{neumannsection}
We now have all the tools needed to prove Theorem \ref{mainthm}. The overall structure of the proof in the obtuse case is nearly identical to that given in \cite{erratum}. The only significant difference is the path of triangles used. Since we limit to an obtuse triangle rather than a right triangle, the argument is somewhat simpler than in the Euclidean case. For the case of right triangles, the argument is new as far as we are aware, though it is simpler than the obtuse case.\\
\indent As in Part 2 of the proof of Theorem \ref{mixed}, we can take a path of negatively curved obtuse (resp. right) triangles that limit to a Euclidean obtuse (resp. right) triangle by fixing a Euclidean triangle $T$ (with edges $e$, $e'$, and $e''$) in the unit disk in $\Rbb^2$ with an obtuse (resp. right) angle at the origin and then varying the metric on $T$ to get a path of obtuse (resp. right) triangles $\Tk$ with curvature $\kappa\in[-1,0]$. Again, we need only prove the theorem for $\kappa=-1$ since it holds in other negative curvatures by simply rescaling the metric. By Corollary \ref{ctspath}, we may re-parameterize this path to $[0,1]\mapsto T_t$ such that there exists accompanying continuous, piecewise real-analytic paths $\mu_t$ and $u_t$ of second Neumann eigenvalues and eigenfunctions of the Laplace operator $-\Delta_{\kappa(t)}$ for $T_t$ and such that $\kappa(0)=0$ and $\kappa(1)=-1$. By Theorem 4.1 of \cite{erratum}, $u_0$ has no critical points, and it takes its global extrema at the acute vertices of $T_0$.

\begin{lem}\label{exactlyone}
    Let $T$ be a constant negative curvature obtuse triangle with a second Neumann eigenfunction $u$. If $u$ has exactly one critical point, then it is stable under perturbation, and either 
    \begin{enumerate}
        \item the coefficient $a_1(u,v)$ in expansion (\ref{neumannexpansion}) at the obtuse vertex $v$ equals zero, or
        \item the eigenfunction $u$ vanishes at an acute vertex of $T$.
    \end{enumerate}    
\end{lem}
\begin{proof}
    The first statement follows from Lemmas \ref{nondegenstable} and \ref{degenimpliesnondegen}. By Lemma \ref{morse}, if $u$ has exactly one critical point, then either each vertex of $T$ is a local extremum of $u$ or $u$ vanishes at an acute vertex. Lemma \ref{vertexmax} implies that $a_1=0$ if the obtuse vertex is a local extremum of $u$.
\end{proof}

For each $t$, let $N(t)$ denote the number of critical points of $u_{t}$.

\begin{proof}[Proof of Theorem \ref{mainthm}]
    Let $S\subseteq [0,1]$ be the set of $t$ such that $N(t)\geq 1$. We show that $S$ is both open and closed in $[0,1]$. Since $0\notin S$ by Theorem 4.1 of \cite{erratum}, this implies that $S=\emptyset$. By Lemmas \ref{nondegenstable}, \ref{stablecp}, and \ref{degenimpliesnondegen}, $S$ is open.\\
    \indent To show that $S$ is closed, let $\{t_n\}$ be a sequence of points in $S$ converging to a point $t\in[0,1]$.\\
    \indent We first show that $S$ is closed if $T_1$ is a right triangle. By Lemma \ref{rightifanythenhyp}, for each $n$, $u_{t_n}$ has a critical point $p_n$ on the hypotenuse of $T$. By passing to a subsequence if necessary, we may suppose that $\{p_n\}$ converges to some point $p$ in the closure of the hypotenuse. If $p$ is in the interior of the hypotenuse, then $p$ is a critical point of $u_{t}$, so $t\in S$. If $p$ is an acute vertex of $T$, then $u_{t}$ vanishes at this vertex by Lemma \ref{convergethenzero}. Then by Corollary \ref{rightexactlyone}, $u_{t}$ has a critical point in the interior of the hypotenuse, so $t\in S$. This completes the case of right triangles. \\
    \indent Now suppose that $T$ is obtuse. By extracting a subsequence if necessary, we may suppose that either $N(t_n)=1$ for all $n$ or $N(t_n)>1$ for all $n$.\\
    \indent First suppose that $N(t_n)=1$ for all $n$. By Lemma \ref{exactlyone}, we may pass to another subsequence such that either the coefficient $a_1$ for $u_{t_n}$ at the obtuse vertex of $T_{t_n}$ vanishes for all $n$ or $u_{t_n}$ vanishes at an acute vertex for all $n$. In the first case, by continuity, the coefficient $a_1$ for $u_{t}$ vanishes at the obtuse vertex of $T_t$. By Lemma \ref{zerocoeffthencp}, $u_{t}$ has a critical point on the longest edge of $T_t$, so $t\in S$. In the second case, by continuity, $u_{t}$ vanishes at an acute vertex. By Lemma \ref{morse}, $u_{t}$ has a local extremum that is not equal to this acute vertex. If one of these local extrema is the obtuse vertex, then Lemmas \ref{vertexmax} and \ref{zerocoeffthencp} imply that $t\in S$. If one of these local extrema is not the obtuse vertex, then $u$ has a non-vertex local extremum that is therefore a critical point, so $t\in S$.\\
    \indent Now suppose that $N(t_n)>1$ for all $n$. By Lemma \ref{twoimpliestwoedges}, for each $n$, there exist critical points $p_n\in e$ and $q_n\in e'$ of $u_{t_n}$ for some pair of distinct edges $e$ and $e'$. Passing to a subsequence if necessary, we may suppose that $\{p_n\}$ converges to a point $p$ in the closure of $e$ and that $\{q_n\}$ converges to a point $q$ in the closure of $e'$. If either of $p$ or $q$ is not equal to a vertex, then this point is a critical point of $u_{t}$ by continuity, and $t\in S$. By Lemma \ref{convtosamevertex}, $p\neq q$. By Lemma \ref{convergethenzero}, if $p\neq q$ are both acute vertices, then $u_{t}$ vanishes at both of these vertices by Lemma \ref{convergethenzero}, contradicting Lemma \ref{nodalsimplearc}. Thus, if $p$ and $q$ are both vertices of $T_t$, then one of them is the obtuse vertex. In this case, Lemma \ref{convergethenzero} implies that $a_1=0$ at the obtuse vertex for $u_{t}$, and Lemma \ref{zerocoeffthencp} in turn implies that $u_{t}$ has a critical point, so $t\in S$. We have therefore established that second Neumann eigenvalues of obtuse hyperbolic triangles have no non-vertex critical points. \\
    \indent Now, note that by Lemma \ref{zeroatacute}, if $u$ is a second Neumann eigenfunction of a non-acute hyperbolic triangle $T$, then $u$ cannot vanish at an acute vertex of $T$. Let $e$ be the longest edge of $T$, and let $L_e$ be a loxodromic Killing field whose axis contains $e$. By Lemma \ref{derivativeneumannvertex}, $\Zcal(L_eu)$ does not contain an arc with an endpoint at either acute vertex. Since $u$ has no critical points, $\Zcal(L_eu)$ also does not have an endpoint in the interior of any edge of $T$. By Lemma \ref{neumannloop}, therefore, $L_eu$ does not vanish in $T$. Replacing $L_e$ by $-L_e$ if necessary, we have $L_eu>0$ in $T$.
\end{proof}

\begin{proof}[Proof of Corollary \ref{simple}]
    Fix an acute vertex $v$ of $T$. By the argument in the last paragraph of the proof of Theorem \ref{mainthm}, the map sending a second Neumann eigenfunction of $T$ to its value at $v$ is an injective linear map from the second Neumann eigenspace to $\Cbb$. The corollary follows from this. 
\end{proof}

\end{document}